\documentclass[10pt]{article}

\usepackage{amssymb,amsmath,amsthm,amsfonts}
\usepackage{graphicx}
\usepackage{comment}
\graphicspath{{./pics/}}
\usepackage[colorlinks,linktocpage,linkcolor=blue]{hyperref}
\usepackage{verbatim}

\newtheorem{thm}{Theorem}[section]
\newtheorem{rmk}{Remark}[section]

\newtheorem{lem}{Lemma}[section]
\newtheorem{ass}{Assumption}[section]

\newtheorem{cor}{Corollary}[section]

\title{An Analysis of Finite Element Approximation in Electrical Impedance Tomography}
\author{Matthias Gehre\footnote{Center for Industrial Mathematics,
University of Bremen, Bremen 28359, Germany (mgehre@math.uni-bremen.de)}
\and Bangti Jin\footnote{Department of Mathematics, University of California, Riverside, 900 University Ave., Riverside, CA 92521, USA (bangti.jin@gmail.com)}
\and Xiliang Lu\footnote{School of Mathematics and Statistics, Wuhan University, Wuhan 430072, P. R. China (xllv.math@whu.edu.cn)}
}
\date{\today}

\begin{document}
\maketitle
\begin{abstract}
We present a finite element analysis of electrical
impedance tomography for reconstructing the conductivity
distribution from electrode voltage measurements by means of Tikhonov regularization.
Two popular choices of the penalty term, i.e., $H^1(\Omega)$-norm smoothness penalty and
total variation seminorm penalty, are considered. A piecewise linear finite element
method is employed for discretizing the forward model, i.e., the complete electrode
model, the conductivity, and the penalty functional. The convergence of the finite element
approximations for the Tikhonov model on both polyhedral and smooth curved domains is
established. This provides rigorous justifications for the ad hoc discretization
procedures in the literature.
\\
\textbf{Keywords}: electrical impedance tomography, finite element approximation,
convergence analysis, Tikhonov regularization
\end{abstract}

\section{Introduction}\label{sect:intro}
Electrical impedance tomography (EIT) is a very popular diffusive imaging modality for
probing internal structures of the concerned object, by recovering its electrical conductivity/permittivity
distribution from voltage measurements on the boundary. One typical experimental setup
is as follows. One first attaches a set of metallic electrodes to the surface of the
object. Then one injects an electric current into the object through these electrodes,
which induces an electromagnetic field inside the object. Finally, one
measures the electric voltages on these electrodes. The procedure is often repeated several times with
different input currents in order to yield sufficient information on the sought-for conductivity
distribution. This physical process can be most accurately described by the complete
electrode model \cite{ChengIsaacsonNewellGisser:1989, SomersaloCheneyIsaacson:1992},
but the simpler continuum model is also frequently employed in simulation studies. The imaging modality has
attracted considerable interest in applications, e.g., in medical imaging, geophysical prospecting,
nondestructive evaluation and pneumatic oil pipeline conveying.

Due to its broad range of prospective applications, a large number of imaging algorithms have been
developed, and have delivered very encouraging reconstructions. These methods essentially utilize the idea of
regularization in diverse forms, in order to overcome the severe ill-posed nature of the imaging task,
and occasionally also the idea of (recursive) linearization to enable
computational tractability. We refer interested readers to the reviews \cite{Borcea:2002,AdGaLi:2011} and the recent references \cite{KaipioKolehmainenSomersaloVauhkonen:2000,SiltanenMuellerIsaacson:2000,RondiSantosa:2001,ChungChanTai:2005,
LukaschewitschMaassPidcock:2003,LechleiterRieder:2006,KnudsenLassasMuellerSiltanen:2009,JinKhanMaass:2010,HarrachSeo:2010}
for a very incomplete list of existing methods. One prominent idea underlying many popular EIT imaging techniques is
Tikhonov regularization with convex variational penalties, e.g., smoothness, total variation
and more recently sparsity constraints \cite{JinKhanMaass:2010,JinMaass:2012}.
These approaches have demonstrated very promising reconstructions for real data; see e.g.
\cite{KarhunenSeppanenLehikoinenMonteiroKaipio:2010,BorsicGrahamAdlerLionheart:2010,GehreKluth:2010} for some recent works.
However, the analysis of such formulations, surprisingly, has not received due
attention, despite its popularity in and relevance to practical applications. We are only aware
of very few works in this direction \cite{RondiSantosa:2001, Rondi:2008, JinMaass:2010}. In the
pioneering works \cite{RondiSantosa:2001,Rondi:2008}, Rondi and Santosa analyzed the existence,
stability and consistency of the Mumford-Shah/total variation formulation. Recently, Jin and Maass
\cite{JinMaass:2010} established the existence, stability, consistency and especially
convergence rates for the conventional Sobolev $H^1$-penalty and sparsity constraints.
These works provide partial theoretical justifications for the
practical usage of related imaging algorithms.

In practice, the numerical implementation of these imaging algorithms inevitably requires discretizing the forward model and the Tikhonov
functional into a finite-dimensional discrete problem. This is often achieved
by the finite element method, due to its versatility for handling general domain
geometries, spatially varying coefficients and solid theoretical underpinnings. However, the
solution to the discrete optimization problem is
different from that to the continuous Tikhonov model due to the discretization errors.
This raises several interesting questions on the discrete
approximations. One fundamental question is about the validity of the discretization procedure:
Does the discrete approximation converge to a solution to the continuous Tikhonov formulation as
the mesh size tends to zero? Since for inverse problems,
small errors in the data/model can possibly cause large deviations in the solution, it is unclear whether
the discretization error induces only small changes on the solution. Hence, the validity of the discretization strategy does
not follow automatically. To the best of our knowledge, the convergence issue
has not been addressed for EIT, despite the fact that such procedures are
routinely adopted in practice. 
However, we note that a closely related problem of compensating the effect of an
imprecise boundary on the resolution of numerical reconstructions has been carefully
studied by Kolehmainen et al \cite{KolehmainenLassasOla:2005,KolehmainenLassasOla:2007},
and several numerical algorithms were also developed.

In this work, we address the convergence issue of finite element approximations.
Specifically, we consider the complete electrode model, and discuss two popular imaging
techniques based on Tikhonov regularization with smoothness/total variation penalties.
These methods have been extensively used in simulation as well as real-world studies;
see \cite{DobsonSantosa:1994,VauhkonenVadasz:1998,VauhkonenVauhkonen:1999,RondiSantosa:2001,ChungChanTai:2005,Rondi:2008,BorsicGrahamAdlerLionheart:2010}
for a very incomplete list.  We shall distinguish two different
scenarios: polyhedral domains and convex smooth curved domains. The former allows exact
triangulation with simplicial elements, whereas the latter invokes domain approximations
and hence the analysis is much more involved. The simpler polyhedral case
serves to illustrate the main ideas of the proof. We remark that for practical applications,
curved domains are very common and their accurate discrete description is
essential for getting reasonable reconstructions, e.g., in imaging human body
\cite{BagshawListon:2003}, and hence it is of immense interest to analyze this case. The
rest of the paper is organized as follows. In Section \ref{sect:prelim}, we describe
the complete electrode model, collect some preliminary regularity results, and recall the
Tikhonov regularization formulation. Then the convergence analysis for polyhedral domains
is discussed in Section \ref{sect:polydom}, and for curved domains in Section
\ref{sect:curvdom}. Finally, some concluding remarks are given in Section \ref{sect:concl}. Throughout,
we shall use $C$ to denote a generic constant, which may differ at different
occurrences but does not depend on the mesh size $h$. We shall also use standard notation
from \cite{Evans:1992a} for the Sobolev spaces $W^{m,p}(\Omega)$.

\section{Preliminaries}\label{sect:prelim}

Here we recapitulate the mathematical formulation of the complete electrode model and
discuss its analytical properties. We shall also briefly describe the continuous Tikhonov
formulation.

\subsection{Complete electrode model}\label{subsect:cem}
According to the comparative experimental studies in
\cite{ChengIsaacsonNewellGisser:1989,SomersaloCheneyIsaacson:1992}, the complete
electrode model (CEM) is currently the most accurate mathematical model for reproducing
EIT experimental data. This is attributed to its faithful modeling of the physics: It takes into
account several important features of real EIT experiments, i.e., discrete nature of the
electrodes, shunting effect and contact impedance effect. We shall briefly recall the
mathematical model and its analytical properties in this part. These properties will be useful in the convergence analysis below.

Let $\Omega$ be an open and bounded domain in $\mathbb{R}^d$ ($d=2,3$) with a Lipschitz
continuous boundary $\Gamma$. We denote the set of electrodes by $\{e_l\}_{l=1}^L$, which
are open connected subsets of the boundary $\Gamma$ and disjoint from each other, i.e.,
$\bar{e}_i\cap\bar{e}_k=\emptyset$ if $i\neq k$. The applied current on the $l$th
electrode $e_l$ is denoted by $I_l$, and the current vector
$I=(I_1,\ldots,I_L)^\mathrm{t}$ satisfies $\sum_{l=1}^LI_l=0$ in view of the law of charge
conservation. Let the space $\mathbb{R}_\diamond^L$ be the subspace of the vector
space $\mathbb{R}^L$ with zero mean, then we have $I\in\mathbb{R}_\diamond^L$.
The electrode voltage $U=(U_1,\ldots,U_L)^\mathrm{t}$ is normalized such that
$U\in\mathbb{R}_\diamond^L$, which represents a grounding condition. Then the
mathematical model for the CEM reads: given the electrical conductivity $\sigma$, positive contact
impedances $\{z_l\}$ and an input current pattern $I\in\mathbb{R}_\diamond^L$, find the
potential $u\in H^1(\Omega)$ and electrode voltage $U\in\mathbb{R}_\diamond^L$ such that
\begin{equation}\label{eqn:cem}
\left\{\begin{aligned}
\begin{array}{ll}
-\nabla\cdot(\sigma\nabla u)=0 & \mbox{ in }\Omega,\\[1ex]
u+z_l\sigma\frac{\partial u}{\partial n}=U_l& \mbox{ on } e_l, l=1,2,\ldots,L,\\[1ex]
\int_{e_l}\sigma\frac{\partial u}{\partial n}ds =I_l& \mbox{ for } l=1,2,\ldots, L,\\ [1ex]
\sigma\frac{\partial u}{\partial n}=0&\mbox{ on } \Gamma\backslash\cup_{l=1}^Le_l.
\end{array}
\end{aligned}\right.
\end{equation}

The physical motivation behind the mathematical model \eqref{eqn:cem} is as follows \cite{SomersaloCheneyIsaacson:1992,CheneyIsaacsonNewell:1999}. The governing
equation is derived under a quasi-static low frequency assumption on the electromagnetic process.
The second line in system \eqref{eqn:cem} models the important contact impedance effect: When
injecting electrical currents into the object, a highly resistive thin layer forms at the
electrode-electrolyte interface (due to certain electrochemical processes), which causes
potential drops across the electrode-electrolyte interface according to Ohm's law. It also incorporates the shunting effect: electrodes
are perfect conductors, and hence the voltage is constant on each electrode. The third line
reflects the fact that the current injected through each electrode is completely confined
therein. In practice, the electrode voltage $U$ can be measured, which is then used for
reconstructing the conductivity distribution $\sigma$.

Due to physical constraint, the conductivity distribution is naturally bounded both from
below and from above by positive constants, hence we introduce the following admissible set
\begin{equation*}
  \mathcal{A}=\{\sigma:\ \lambda\leq \sigma(x)\leq \lambda^{-1}\mbox{ a.e. } x\in\Omega\},
\end{equation*}
for some $\lambda\in(0,1)$. We shall endow the set with $L^r(\Omega)$ norms, $r\geq1$.

We denote by $\mathbb{H}$ the product space $H^1(\Omega)\otimes
\mathbb{R}_\diamond^L$ with its norm defined by
\begin{equation*}
  \|(u,U)\|_{\mathbb{H}}^2 = \|u\|_{H^1(\Omega)}^2 + \|U\|_{\mathbb{R}^L}^2.
\end{equation*}

A convenient equivalent norm on the space $\mathbb{H}$ is given in the next lemma
\cite{SomersaloCheneyIsaacson:1992}.
\begin{lem}\label{lem:normequiv}
On the space $\mathbb{H}$, the norm $\|\cdot\|_\mathbb{H}$ is equivalent to the norm
$\|\cdot\|_{\mathbb{H},*}$ defined by
\begin{equation*}
  \|(u,U)\|_{\mathbb{H},*}^2 = \|\nabla u\|_{L^2(\Omega)}^2 + \sum_{l=1}^L\|u-U_l\|_{L^2(e_l)}^2.
\end{equation*}
\end{lem}

The weak formulation of the model \eqref{eqn:cem} reads: find $(u,U)\in \mathbb{H}$ such
that
\begin{equation}\label{eqn:cemweakform}
  \int_\Omega\sigma \nabla u \cdot \nabla v dx +\sum_{l=1}^Lz_l^{-1}\int_{e_l}(u-U_l)(v-V_l)ds =
  \sum_{l=1}^LI_lV_l\quad \forall (v,V)\in \mathbb{H}.
\end{equation}

Now for any fixed $\sigma\in\mathcal{A}$, the existence and uniqueness of a solution $(u,U)\equiv(u(\sigma),U(\sigma))
\in\mathbb{H}$ to the weak formulation \eqref{eqn:cemweakform} follows directly from
Lemma \ref{lem:normequiv} and Lax-Milgram theorem, and further, it
depends continuously on the input current pattern $I$ \cite{SomersaloCheneyIsaacson:1992}. The
next result presents an improved regularity of the solution $(u(\sigma),U(\sigma))$ to
system \eqref{eqn:cem}. It can be derived from the Neumann analogue \cite{Groger:1989,GallouetMonier:1999} of Meyers'
celebrated gradient estimates \cite{Meyers:1963}; see \cite{JinMaass:2010} for details.
\begin{thm}\label{thm:cemreg}
Let $\lambda\in(0,1)$, and $\sigma(x)\in[\lambda,\lambda^{-1}]$ almost everywhere. Then
there exists a constant $Q(\lambda,d)>2$, which depends only on the domain $\Omega$, the spatial
dimension $d$ and the constant $\lambda$, such that for any $q\in(2,Q(\lambda,d))$, the solution
$(u(\sigma),U(\sigma))\in \mathbb{H}$ to system \eqref{eqn:cem} satisfies the following estimate
\begin{equation*}
  \|u\|_{W^{1,q}(\Omega)} \leq C\|I\|,
\end{equation*}
where the constant $C=C(\Omega,d,\lambda,q)$.
\end{thm}

\begin{rmk}
The parameter $Q$ depends on the regularity of the domain $\Omega$. If the domain $\Omega$ is of class
$C^1$, then $Q(\lambda,d)\rightarrow\infty$ as $\lambda\rightarrow1$ \cite{Groger:1989}. For a general
Lipschitz domain, e.g., polyhedrons, there also always exists some $Q(\lambda,d)>2$
for any $\lambda<1$, cf. \cite[Sect. 5]{Groger:1989}.
\end{rmk}


The next result shows that the parameter-to-state map $\sigma\to(u(\sigma),U(\sigma))\in
\mathbb{H}$ is continuous with respect to $L^r(\Omega)$ topology on the
admissible set $\mathcal{A}$.
\begin{lem}\label{lem:contpoly}
Let the sequence $\{\sigma_n\}\subset \mathcal{A}$ converge to some
$\sigma^\ast\in\mathcal{A}$ in $L^r(\Omega),\ r\geq1$. Then the sequence of the solutions
$\{(u(\sigma_n),U(\sigma_n))\}$ converges to $(u(\sigma^\ast),U(\sigma^\ast))$ in $\mathbb{H}$.
\end{lem}
\begin{proof}
It follows from the weak formulations of the solutions $(u_n,U_n)\equiv
(u(\sigma_n),U(\sigma_n))$ and $(u^\ast,U^\ast)\equiv (u(\sigma^\ast),U(\sigma^\ast))$
(cf. \eqref{eqn:cemweakform}) that for all $(v,V)\in \mathbb{H}$
\begin{equation*}
  \int_\Omega\sigma_n\nabla(u^\ast-u_n)\cdot\nabla vdx +\int_\Omega(\sigma^\ast-\sigma_n)\nabla u^\ast\cdot\nabla vdx
  +\sum_{l=1}^Lz_l^{-1}\int_{e_l}(u^\ast-u_n-U^\ast_l+U_{n,l})(v-V_l)ds = 0.
\end{equation*}
Upon setting the test function $(v,V)$ to $(u^\ast-u_n,U^\ast-U_{n})\in\mathbb{H}$ in this
identity, and using Theorem \ref{thm:cemreg} and the generalized H\"{o}lder's
inequality, we derive
\begin{equation*}
  \begin{aligned}
     & \min(\lambda,\{z_l^{-1}\})\left(\|\nabla(u^\ast-u_n)\|_{L^2(\Omega)}^2 + \sum_{l=1}^L\|u^\ast-u_n-U_l^\ast+U_{n,l}\|_{L^2(e_l)}^2\right)\\
     \leq&\int_\Omega\sigma_n|\nabla(u^\ast-u_n)|^2dx+\sum_{l=1}^Lz_l^{-1}\int_{e_l}|u^\ast-u_n-U^\ast_l+U_{n,l}|^2ds\\
     =&-\int_\Omega(\sigma^\ast-\sigma_n)\nabla u^\ast\cdot\nabla(u^\ast-u_n)dx\\
     \leq&\|\sigma^\ast-\sigma_n\|_{L^p(\Omega)}\|\nabla u^\ast\|_{L^q(\Omega)}\|\nabla (u^\ast-u_n)\|_{L^2(\Omega)}\\
     \leq&\|\sigma^\ast-\sigma_n\|_{L^p(\Omega)}\|\nabla u^\ast\|_{L^q(\Omega)}\|(u^\ast-u_n,U^\ast-U_n)\|_{\mathbb{H}},
  \end{aligned}
\end{equation*}
where the exponent $q\in(2,Q(\lambda,d))$ is from Theorem \ref{thm:cemreg} and the
exponent $p$ satisfies $p^{-1}+q^{-1}=2^{-1}$. The desired assertion follows
immediately if $r\geq p$. In the case $r<p$, we exploit the $L^\infty(\Omega)$ bound of the admissible
set $\mathcal{A}$, i.e.,
\begin{equation*}
  \int_\Omega|\sigma^\ast-\sigma_n|^pdx\leq\lambda^{r-p}\int_\Omega|\sigma^\ast-\sigma_n|^{r}dx.
\end{equation*}
This together with Lemma \ref{lem:normequiv} shows the desired assertion.
\end{proof}

\subsection{Tikhonov regularization}\label{subsect:tikh}

The EIT inverse problem is to reconstruct an approximation to the physical conductivity
$\sigma^\dagger$ from noisy measurements $U^\delta$ of the electrode voltage
$U(\sigma^\dagger)$. It is severely ill-posed in the sense
that small errors in the data can lead to very large deviations in the
solutions. Therefore, some sort of regularization is beneficial, and it is
usually incorporated into EIT imaging algorithms, either implicitly or explicitly, in order to
yield stable yet accurate conductivity images. One of the most popular and
successful techniques is the standard Tikhonov regularization. It amounts to minimizing
the celebrated Tikhonov functional
\begin{equation}\label{eqn:tikh}
   \min_{\sigma\in\mathcal{A}}\left\{J(\sigma) = \tfrac{1}{2} \|U(\sigma)-U^\delta\|^2 + \alpha \Psi(\sigma)\right\},
\end{equation}
and then taking the minimizer, denoted by $\sigma_\alpha^\delta$, as an approximation to the
sought-for physical conductivity $\sigma^\dagger$.
Here the first term in the functional $J$ captures the information encapsulated in the data $U^\delta$.
For simplicity, we consider only one dataset, and the adaptation to multiple datasets is
straightforward. The scalar $\alpha>0$ is known as a regularization parameter, and controls
the tradeoff between the two terms. The second term $\Psi(\sigma)$ in the functional $J$ imposes a priori regularity
knowledge (smoothness) on the expected conductivity distributions.
Two most commonly used penalties are $\Psi(\sigma)=\tfrac{1}{2}\|\sigma\|_{H^1(\Omega)}^2$
and $\Psi(\sigma)=|\sigma|_\mathrm{TV(\Omega)}$ in the space of functions with bounded variation, i.e.,
\begin{equation*}
  \mathrm{BV}(\Omega)=\{v\in L^1(\Omega): \|v\|_{\mathrm{BV}(\Omega)}<\infty\},
\end{equation*}
where $\|v\|_{\mathrm{BV}(\Omega)}=\|v\|_{L^1(\Omega)}+|v|_{\mathrm{TV}(\Omega)}$ with the
total variation semi-norm $|v|_{\mathrm{TV}(\Omega)}=\int_\Omega |Dv|$ defined by
\begin{equation*}
  \int_\Omega|Dv| = \sup_{\substack{g\in (C_0^1(\Omega))^d\\
   |g(x)|\leq 1}}\int_\Omega v\mathrm{div}(g)dx.
\end{equation*}
Here the $H^1(\Omega)$-smoothness approach allows reconstructing conductivity distributions that are globally smooth,
which often retains well their main features, whereas the total variation approach is well suited
to discontinuous, especially piecewise constant, conductivity distributions \cite{RondiSantosa:2001,
ChungChanTai:2005}. These two approaches represent the most popular EIT imaging techniques in
practice. Theoretically, the existence and consistency of the continuous
model \eqref{eqn:tikh} for the total variation and smoothness penalty have recently been
established in \cite{Rondi:2008} and \cite{JinMaass:2010}, respectively, where in the latter
work convergence rates for the smoothness and sparsity constraints were also provided.

One useful tool in the convergence analysis is the following embedding results \cite{Evans:1992a,AttouchButtazzoMichaille:2006}.
\begin{lem} \label{lem:embed}
The spaces $H^1(\Omega)$ and $\mathrm{BV}(\Omega)$ have the following embedding properties:
\begin{itemize}
  \item[(a)] The space $H^1(\Omega)$ embeds compactly into $L^{p}(\Omega)$ for $p<\infty$ if $d=2$ and $p<6$ if $d=3$.
  \item[(b)] The space $\mathrm{BV}(\Omega)$ embeds compactly into $L^{p}(\Omega)$ for $p<\frac{d}{d-1}$.
\end{itemize}
\end{lem}

A direct consequence of Lemmas \ref{lem:contpoly} and \ref{lem:embed} is the weak
continuity. The concept of weak convergence in the $\mathrm{BV}$ space used below
follows \cite[Definition 10.1.2]{AttouchButtazzoMichaille:2006}.
\begin{cor}\label{cor:weakcont}
Let the sequence $\{\sigma_n\}\subset \mathcal{A}$ converge to some
$\sigma^\ast\in\mathcal{A}$ weakly in either $H^1(\Omega)$ or
$\mathrm{BV}(\Omega)$. Then the sequence of the solutions
$\{(u(\sigma_n),U(\sigma_n))\}$ converges strongly to $(u(\sigma^\ast),U(\sigma^\ast))$
in $\mathbb{H}$.
\end{cor}

Corollary \ref{cor:weakcont} implies that the forward parameter-to-state map is weakly sequentially closed, and in
view of the classical nonlinear Tikhonov regularization theory \cite{EnglKunischNeubauer:1989}, this
directly yields the the existence of a minimizer and its stability. We will also need the following
density result for the space $\mathrm{BV}(\Omega)$; see \cite[Lemma 3.3]{ChenZou:1999} for a proof:
\begin{lem}\label{lem:density}
Let $g\in \mathrm{BV}(\Omega)$. Then for any $\epsilon>0$, there exists a function $g_\epsilon\in C^\infty(\overline{\Omega})$ such that
\begin{equation*}
  \int_\Omega |g-g_\epsilon|dx<\epsilon,\quad \left|\int_\Omega|\nabla g_\epsilon|dx-\int_\Omega|Dg|\right|<\epsilon.
\end{equation*}
\end{lem}

In order to obtain conductivity images from a computer implementation of the Tiknonov approach,
one necessarily needs to discretize the forward
problem \eqref{eqn:cem} and the Tikhonov functional \eqref{eqn:tikh} by restricting the
admissible conductivities to a certain finite-dimensional subspace. In practice, this is
usually achieved by the finite element method due to its solid theoretical foundation and
versatility for handling general domain geometries, as often occur in practical
situations. The main goal of the present study is to provide theoretical justifications
for such procedures. We shall discuss two scenarios separately: polyhedral domains and
(convex) smooth curved domains in Sections \ref{sect:polydom} and \ref{sect:curvdom}, respectively.

\section{Convergence for polyhedral domains}\label{sect:polydom}

In this part, we discuss the case of polyhedral domains. Let $\Omega$ be an open bounded
polyhedral domain. To discretize the imaging problem, we first
triangulate the domain $\Omega$. Let $\mathcal{T}_h$ be a family of shape regular,
quasi-uniform triangulation of the domain $\Omega$, with the mesh consisting of
simplicial elements. The mesh size (the radius of the smallest circle/sphere
circumscribing each element) of the mesh $\mathcal{T}_h$ is denoted by $h$. On the mesh
$\mathcal{T}_h$, we define a continuous piecewise linear finite element space
\begin{equation*}
   V_h = \left\{v\in C(\overline{\Omega}): v|_T\in P_1(T)\ \forall T\in\mathcal{T}_h\right\},
\end{equation*}
where the space $P_1(T)$ consists of all linear functions on the element $T$. The same space
$V_h$ is used for approximating both the potential $u$ and the conductivity $\sigma$. Nonetheless,
we observe that in practice, it is possible to employ different meshes for the potential and the conductivity,
for which the analysis below remains valid upon minor modifications. The
use of piecewise linear finite elements is especially popular since the data
(conductivity and boundary conditions) has only limited regularity.

With the space $V_h$, we can define two important operators: the canonical nodal
interpolation operator $\mathcal{I}_h: C(\overline{\Omega})\rightarrow V_h$ and the $H^1$-projection
operator $\mathcal{R}_h: H^1(\Omega)\rightarrow V_h$ defined by
\begin{equation*}
  \int_\Omega \nabla \mathcal{R}_hu\cdot\nabla vdx +\int_\Omega \mathcal{R}_huvdx
    = \int_\Omega \nabla u\cdot\nabla vdx + \int_\Omega uvdx\quad \forall v\in V_h.
\end{equation*}
It is well known \cite{Ciarlet:2002} that the operators $\mathcal{I}_h$ and
$\mathcal{R}_h$ satisfy for any $p>d$
\begin{equation}\label{eqn:basicest}
   \begin{aligned}
     \lim_{h\rightarrow0}\|v-\mathcal{I}_hv\|_{L^\infty(\Omega)}&=0\quad \forall v\in W^{1,p}(\Omega),\\
     \lim_{h\rightarrow0}\|v-\mathcal{R}_hv\|_{H^1(\Omega)}&=0\quad \forall v\in H^1(\Omega).
   \end{aligned}
\end{equation}

Now we can describe the finite element approximation scheme. First, we approximate the
forward map $(u(\sigma),U(\sigma))\in \mathbb{H}$ by $(u_h,U_h)\equiv(u_h(\sigma_h),
U_h(\sigma_h))\in \mathbb{H}_h\equiv V_h\otimes\mathbb{R}_\diamond^L$ defined by
\begin{equation}\label{eqn:dispoly}
  \int_\Omega \sigma_h\nabla u_h\cdot\nabla v_h dx +\sum_{l=1}^Lz_l^{-1}\int_{e_l}(u_h-U_{h,l})(v_h-V_l)ds = \sum_{l=1}^LI_lV_l\quad \forall
  (v_h,V)\in V_h\otimes \mathbb{R}_\diamond^L,
\end{equation}
where the (discretized) conductivity $\sigma_h$ lies in the discrete admissible set
\begin{equation*}
  \mathcal{A}_h=\{\sigma_h\in V_h:\lambda\leq \sigma_h\leq \lambda^{-1}\ \mbox{ a.e. } \Omega\}=\mathcal{A}\cap V_h.
\end{equation*}
Then the discrete optimization problem reads
\begin{equation}\label{eqn:discopt}
  \min_{\sigma_h\in\mathcal{A}_h}\left\{J_h(\sigma_h) = \tfrac{1}{2}\|U_h(\sigma_h)-U^\delta\|^2 + \alpha\Psi(\sigma_h)\right\},
\end{equation}
where the discrete penalty functional $\Psi(\sigma_h)$ is given by
$\Psi(\sigma_h)= \tfrac{1}{2}\|\sigma_h\|_{H^1(\Omega)}^2$ and $\Psi(\sigma_h) = |\sigma_h|_{\mathrm{TV}(\Omega)}$
for the smoothness and total variation penalty, respectively.

\begin{rmk}
In practice, even though the domain $\Omega$ is polyhedral, the electrode surfaces $\{e_l\}$ can still be curved,
and this calls for the approximation of the surfaces $\{e_l\}$ by polyhedral surfaces in the discrete variational
formulation. However, we defer relevant discussions to Section \ref{sect:curvdom}.
\end{rmk}

We observe that, due to the use of linear finite elements, the box constraint on
$\sigma_h$ reduces to that on the nodal values, which greatly facilitates the solution of the resulting
discrete optimization problem. Since the set $\mathcal{A}_h$ is finite dimensional
and uniformly bounded, the compactness and the norm equivalence of finite-dimensional spaces
immediately yields the existence of a minimizer $\sigma_h^\ast\in\mathcal{A}_h$ to the
discrete functional $J_h(\sigma_h)$ over the discrete admissible set $\mathcal{A}_h$ for
any $h>0$.

One basic question is whether the sequence $\{\sigma_h^\ast\}$ of discrete minimizers
converges to a minimizer of the continuous functional $J(\sigma)$ as the mesh size $h$
tends to zero. This issue is concerned with the validity of the approximation procedure, and
hence it is of significant practical interest.
To this end, we shall first establish a discrete analogue of Lemma \ref{lem:contpoly} on
the approximation property of the discrete parameter-to-state map $\sigma_h\mapsto (u_h(\sigma_h),U_h(\sigma_h))$ to the
continuous counterpart $\sigma\mapsto (u(\sigma),U(\sigma))$. The lemma will play a crucial role
in establishing the desired convergence for polyhedral domains.
\begin{lem}\label{lem:polyhed}
Let the sequence $\{\sigma_h\}_{h>0}\subset \mathcal{A}_h\subset\mathcal{A}$ converge
in $L^r(\Omega), \ r\geq 1$, to some $\sigma\in\mathcal{A}$ as $h$ tends to zero. Then the sequence
of finite element approximations $\{(u_h(\sigma_h),U_h(\sigma_h))\}_{h>0}$ converges
to $(u(\sigma),U(\sigma))$ in $\mathbb{H}$ as $h$ tends to zero.
\end{lem}
\begin{proof}
First recall the weak formulation of $(u,U)\equiv (u(\sigma),U(\sigma))$ and $(u_h,U_h)
\equiv(u_h(\sigma_h),U_h(\sigma_h))$ in \eqref{eqn:cemweakform} and \eqref{eqn:dispoly}, respectively. 
It follows from Lax-Milgram theorem that both $(u,U)$ and $(u_h,U_h)$ are uniformly bounded in $\mathbb{H}$.
By setting the test functions $(v,V)$ and $(v_h,V)$ in identities
\eqref{eqn:cemweakform} and \eqref{eqn:dispoly} to
$(\mathcal{R}_hu-u_h,U-U_h)\in\mathbb{H}_h\subset \mathbb{H}$ and then subtracting them,
we deduce
\begin{equation*}
   \begin{aligned}
      &\int_\Omega\sigma_h|\nabla(u-u_h)|^2dx+\sum_{l=1}^Lz_l^{-1}\int_{e_l}|u-u_h-U_l+U_{h,l}|^2ds\\
      =&-\int_\Omega(\sigma-\sigma_h)\nabla u\cdot\nabla(\mathcal{R}_hu-u_h)dx +
       \int_\Omega\sigma_h\nabla(u-u_h)\cdot\nabla(u-\mathcal{R}_h u)dx\\
       &+ \sum_{l=1}^Lz_l^{-1}\int_{e_l}(u-\mathcal{R}_hu)(u-u_h-U_l+U_{h,l})ds:=I+II+III.
   \end{aligned}
\end{equation*}
It suffices to estimate the three terms ($I$, $II$ and $III$) on the right hand side. For
the first term $I$, the generalized H\"{o}lder's inequality gives
\begin{equation*}
 | I |\leq \|\sigma-\sigma_h\|_{L^p(\Omega)}\|\nabla u\|_{L^q(\Omega)}\|\nabla (\mathcal{R}_hu-u_h)\|_{L^2(\Omega)},
\end{equation*}
where the exponent $q>2$ is from Theorem \ref{thm:cemreg}, and the exponent $p$ satisfies
$p^{-1}+q^{-1}=2^{-1}$. Further, we note that
\begin{equation*}
   \begin{aligned}
      \|\nabla(\mathcal{R}_hu-u_h)\|_{L^2(\Omega)}&\leq \|\nabla \mathcal{R}_hu\|_{L^2(\Omega)}+\|\nabla u_h\|_{L^2(\Omega)}\\
      &\leq C(\|u\|_{H^1(\Omega)}+\|u_h\|_{H^1(\Omega)})<C.
   \end{aligned}
\end{equation*}
By repeating the proof in Lemma \ref{lem:contpoly}, we deduce that the first term $I\rightarrow0$
as $h$ tends to zero. For the second term $II$, we deduce from the uniform bound of the discrete
admissible set $\mathcal{A}_h$ that
\begin{equation*}
   \begin{aligned}
      II &\leq \|\sigma_h\|_{L^\infty(\Omega)}\|\nabla(u-u_h)\|_{L^2(\Omega)}\|\nabla(u-\mathcal{R}_hu)\|_{L^2(\Omega)}\\
         & \leq \lambda^{-1}\|\nabla(u-u_h)\|_{L^2(\Omega)}\|\nabla(u-\mathcal{R}_hu)\|_{L^2(\Omega)},
   \end{aligned}
\end{equation*}
which tends to zero in light of the approximation property of the operator $\mathcal{R}_h$ in
\eqref{eqn:basicest} and uniform boundedness of $\|\nabla(u-u_h)\|_{L^2(\Omega)}$.
The third term $III$ follows analogously from the trace theorem \cite{Evans:1992a}.
These three estimates together with Lemma \ref{lem:normequiv} yield the desired assertion.
\end{proof}

Now we can state the first main result, i.e., the convergence of finite element
approximations $\{\sigma_h^\ast\}$ on polyhedral domains.
\begin{thm}\label{thm:polyconv}
Let $\{\sigma_h^\ast\in\mathcal{A}_h\}_{h>0}$ be a sequence of minimizers to the discrete optimization
problem \eqref{eqn:discopt}. Then it contains a subsequence convergent to a minimizer of problem \eqref{eqn:tikh} as $h$ tends to zero.
\begin{itemize}
  \item[(a)] The convergence is weakly in $H^1(\Omega)$, if $\Psi(\sigma_h)=\tfrac{1}{2}\|\sigma_h\|_{H^1(\Omega)}^2$;
  \item[(b)] The convergence is in $L^1(\Omega)$, if $\Psi(\sigma_h)=|\sigma_h|_{\mathrm{TV}(\Omega)}$.
\end{itemize}
\end{thm}
\begin{proof}
First we note that the constant function $\sigma_h\equiv1$ belongs to the discrete admissible set
$\mathcal{A}_h$ for all $h$. The minimizing property of $\sigma_h^\ast$ indicates
that the sequence of functional values $\{J_h(\sigma_h^\ast)\}$ is uniformly bounded. Thus the
sequence $\{\Psi(\sigma_h^\ast)\}$ is uniformly bounded, and there exists a
subsequence, again denoted by $\{\sigma_h^\ast\}$, and some $\sigma^\ast\in\mathcal{A}$,
such that $\sigma_h^\ast\rightarrow \sigma^\ast$ weakly either in $H^1(\Omega)$
or $\mathrm{BV}(\Omega)$. By Lemma \ref{lem:embed},
we have $\sigma_h^\ast\rightarrow \sigma^\ast$ in $L^1(\Omega)$, which together with
Lemma \ref{lem:polyhed} implies
\begin{equation*}
   (u_h(\sigma_h^\ast),U(\sigma_h^\ast)) \rightarrow (u(\sigma^\ast),U(\sigma^\ast))\quad \mbox{ in } \mathbb{H}\mbox{ as } h\rightarrow0.
\end{equation*}
Meanwhile, the weak lower semicontinuity of norms implies
$\Psi(\sigma^\ast)\leq \liminf_{h\rightarrow0}\Psi(\sigma_h^\ast).$
Altogether, we derive
\begin{equation}\label{eqn:lsc}
  \begin{aligned}
    J(\sigma^\ast) &= \tfrac{1}{2}\|U(\sigma^\ast)-U^\delta\|^2 + \alpha\Psi(\sigma^\ast)\\
     &\leq \lim_{h\rightarrow0}\tfrac{1}{2}\|U_h(\sigma_h^\ast)-U^\delta\|^2 + \liminf_{h\rightarrow0}\alpha\Psi(\sigma_h^\ast)\\
     &\leq \liminf_{h\rightarrow0}\left(\tfrac{1}{2}\|U_h(\sigma_h^\ast)-U^\delta\|^2 + \alpha\Psi(\sigma_h^\ast)\right)
     =\liminf_{h\rightarrow0}J_h(\sigma_h^\ast)
  \end{aligned}
\end{equation}

Now we discuss the two penalties separately. First we consider the case $\Psi(\sigma)=\tfrac{1}{2}
\|\sigma\|_{H^1(\Omega)}^2$. For any $\sigma\in\mathcal{A}$, the density
of the space $C^\infty(\overline{\Omega})$ in the space $H^1(\Omega)$ \cite{Evans:1992a}
implies the existence of a sequence
$\{\sigma^\epsilon\}\subset C^\infty(\overline{\Omega})\cap\mathcal{A}$ such that
\begin{equation}\label{eqn:denpoly}
\lim_{\epsilon\rightarrow0^+}\|\sigma^\epsilon-\sigma\|_{H^1(\Omega)}=0.
\end{equation}
The minimizing property of $\sigma_h^\ast$ gives $J_h(\sigma_h^\ast)\leq J_h(\mathcal{I}_h\sigma^\epsilon)$
for any $\epsilon>0$. Letting $h$ to zero, and appealing to the property of interpolation operator $\mathcal{I}_h$,
Lemma \ref{lem:polyhed} and \eqref{eqn:lsc} yield $J(\sigma^\ast)\leq J(\sigma^\epsilon)$.
Since $\epsilon$ is arbitrary, by letting $\epsilon$ to zero, noting the approximation property of the
sequence $\sigma^\epsilon$ in \eqref{eqn:denpoly}, and the continuity result in Lemma
\ref{lem:contpoly}, we deduce that $J(\sigma^\ast)\leq J(\sigma)$ for any
$\sigma\in\mathcal{A}$. This shows the desired assertion for $\Psi(\sigma)=\tfrac{1}{2}\|\sigma\|_{H^1(\Omega)}^2$.

Next we consider the case $\Psi(\sigma)=|\sigma|_{\mathrm{TV}(\Omega)}$.
For any $\sigma\in\mathcal{A}$, Lemma \ref{lem:density} implies the existence of a sequence
$\{\sigma^\epsilon\}\subset C^\infty(\overline{\Omega})$ such that
$\int_\Omega |\sigma^\epsilon-\sigma|dx<\epsilon$ and $\left|\int_\Omega|\nabla \sigma^\epsilon|dx
-\int_\Omega|D\sigma|\right|<\epsilon$. Next we define $\tilde{\sigma}_\epsilon =P_{[c_0,c_1]}
\sigma^\epsilon$, where the operator $P_{[c_0,c_1]}$ denotes pointwise projection.
Since $\nabla \tilde{\sigma}^\epsilon =\nabla \sigma^\epsilon\chi_{\Omega_\epsilon}$ (with
the set $\Omega_\epsilon = \{x\in\Omega: c_0\leq \sigma^\epsilon\leq c_1\}$), which is
uniformly bounded, and thus $\tilde{\sigma}^\epsilon\in \mathcal{A}\cap W^{1,\infty}(\Omega)$.
With the choice $\sigma_h=\mathcal{I}_h\tilde{\sigma}^\epsilon\in V_h$, the minimizing property of
$\sigma_h^\ast\in\mathcal{A}_h$ gives $J_h(\sigma_h^\ast)\leq J_h(\mathcal{I}_h\tilde{\sigma}^\epsilon)$
for any $\epsilon>0$. By the approximation property of the operator $\mathcal{I}_h$ in
\eqref{eqn:basicest} and the fact that $\tilde{\sigma}^\epsilon\in W^{1,\infty}(\Omega)$,
we deduce
\begin{equation*}
  \lim_{h\rightarrow0}\mathcal{I}_h\tilde{\sigma}^\epsilon = \tilde{\sigma}^\epsilon \quad \mbox{ in } W^{1,1}(\Omega).
\end{equation*}
Letting $h$ to zero, and appealing to Lemma \ref{lem:polyhed} and
\eqref{eqn:lsc} yield $J(\sigma^\ast)\leq J(\tilde{\sigma}^\epsilon)$. We observe
the following approximation properties of the sequence $\tilde{\sigma}^\epsilon$,
\begin{equation*}
  \begin{aligned}
   \int_\Omega |\nabla\tilde{\sigma }^\epsilon|dx &= \int_{\Omega_\epsilon}|\nabla\sigma^\epsilon|dx \leq \int_\Omega
    |\nabla \sigma^\epsilon|dx \leq \int_\Omega |D\sigma|+\epsilon,\\
    \int_\Omega |\tilde{\sigma}^\epsilon-\sigma|dx &\leq \int_\Omega|\sigma^\epsilon-\sigma|dx<\epsilon,
  \end{aligned}
\end{equation*}
where the last line follows from the contraction property of the operator $P_{[c_0,c_1]}$.
By letting $\epsilon$ to zero and the continuity result in Lemma
\ref{lem:contpoly}, we deduce $J(\sigma^\ast)\leq J(\sigma)$ for any
$\sigma\in\mathcal{A}$, i.e., $\sigma^\ast$ is indeed a minimizer to the functional
$J(\sigma)$. This concludes the proof of the theorem.
\end{proof}

\begin{rmk}
A close inspection of the proof of Theorem \ref{thm:polyconv} indicates with minor
modifications the result holds also for the continuum model, provided that the input
current $j$ satisfies a certain regularity condition so that an analogue of Theorem
\ref{thm:cemreg} is valid, cf. \cite[Appendix A]{JinKhanMaass:2010}. Further, the
analysis can be easily adapted to multi-parameter models, e.g.,
$\Psi(\sigma)=\tfrac{1}{2}\|\sigma\|_{H^1(\Omega)}^2+\gamma \|\sigma\|_{L^1(\Omega)}$.
\end{rmk}

\section{Convergence for smooth curved domains}\label{sect:curvdom}

Now we turn to the convergence analysis of finite element approximations on convex smooth curved
domains. In the finite element literature, there are several different ways to treat curved domains, e.g.,
isoparametric elements \cite{Ciarlet:2002} and curved elements \cite{Bernardi:1989}.
In EIT imaging algorithms, we usually approximate the domain $\Omega$ with a polyhedral domain $\Omega_h$ (with its
boundary denoted by $\Gamma_h$), and solve the forward problem \eqref{eqn:cem} directly
on the polyhedral domain $\Omega_h$, with the resulting solution taken as the
desired approximation. That is, all computations are performed on a polyhedral domain
$\Omega_h$. Such a discretization strategy has been routinely employed in the implementation of
EIT imaging techniques, but to the best of our knowledge, it has not been rigorously justified.

Throughout, the triangulation $\mathcal{T}_h$ is shape regular and quasi-uniform and it
consists of simplicial elements, and the finite element space $V_h$ is the canonical piecewise
linear finite element space defined on $\mathcal{T}_h$. Further, we make the following assumption
on the domain $\Omega$ and the polyhedral approximation $\Omega_h$ (with their boundaries being
$\Gamma$ and $\Gamma_h$, respectively). The finite element space $V_h$ will be used to discretize
both the forward model and the conductivity distribution.
\begin{ass} \label{ass:poly}
The domain $\Omega$ is convex with a $C^2$ boundary $\Gamma$. The approximating polyhedral domain $\Omega_h$ is also convex, and
the vertices of $\Gamma_h$ are on the boundary $\Gamma$.
\end{ass}

\begin{rmk}
The convexity and smoothness in Assumption \ref{ass:poly} is mainly for 3D domains. In the 2D case,
the discussions below work for domains with a piecewise smooth boundary; see e.g.
\cite{BrambleKing:1994}. That is, the convexity of the domain is not required then.
\end{rmk}

To ease the exposition, we introduce some further notation. By Assumption \ref{ass:poly},
clearly there holds the relation $\Omega_h\subset\Omega$. Due to the convexity of the domain $\Omega_h$,
we can define a projection operator $\phi_h$ by
\begin{equation*}
  \begin{aligned}
  \phi_h: &\ \ \overline{\Omega}\backslash{\Omega_h} \mapsto \Gamma_h,\\
  &\ \ \phi_h(x) = 
  \text{argmin}_{z\in \Gamma_h} |x-z|.
  \end{aligned}
\end{equation*}
We denote by $\phi_h^1$ and $\phi_h^2$ the map $\phi_h$ restricted to the interior
domain $\Omega\backslash\overline{\Omega}_h$ and the boundary $\Gamma$, respectively, i.e.,
$\phi_h^1 = \phi_h|_{\Omega\backslash\overline{\Omega}_h}$ and $\phi_h^2 = \phi_h|_{\Gamma}$.
Next let $S_h$ be any finite element surface (of the triangulation $\mathcal{T}_h)$
of the polyhedral domain $\Omega_h$ lying on $\Gamma_h$,
and $T_h$ be the finite element to which $S_h$ belongs. We denote the pair by
$(S_h,T_h)$, and by $(\phi_h^1)^{-1}(S_h)$ the preimage of $S_h$ under $\phi_h^1$.
Then there holds $\Omega\backslash\overline{\Omega}_h \subset \cup_{S_h}(\phi_h^1)^{-1}(S_h)$.
Further for any interior point $x_h$ to the surface $S_h$ (understood elementwise), we
can define a unit outward normal vector $n_{x_h}$ through $x_h$, which is perpendicular
to $S_h$. Due to the convexity of the domain $\Omega$, the outward normal vector $n_{x_h}$
intersects the boundary $\Gamma$ uniquely at a point $x$. This defines a map
\begin{equation*}
  \begin{aligned}
    \psi_h:&\ \ \cup_{S_h}\mbox{int}{S}_h \mapsto \Gamma,\\
           &\ \ \psi_h(x_h) = x.
  \end{aligned}
\end{equation*}
Furthermore, by the convexity of the domain $\Omega_h$, $\Omega_h$ lies on only one side
of any plane that contains $S_h$, and therefore, $\phi_h^2(\psi_h(x_h)) = x_h$,
$x_h\in\mathrm{int}S_h$. For any surface patch $e\subset \Gamma$, there holds
\begin{equation}\label{equ:e1}
  \phi_h^2(e) = \cup {S_{e,h}},
\end{equation}
where $S_{e,h} = S_h\cap \phi_h^2(e)$.
Further, we denote the subset $\cup_{x\in\mbox{int}S_{e,h}}\psi_h(x)\subset e$ by $\varrho_h(e)$, which will
be taken as an approximation to the surface electrode $e$ in the convergence analysis.

The following lemma provides the crucial estimates on the domain approximation under Assumption \ref{ass:poly}. These
estimates are crucial to subsequent convergence analysis.
\begin{lem}\label{lem:assump}
Let Assumption \ref{ass:poly} be fulfilled. Then for any small $h$, the following statements hold.
\begin{itemize}
\item[(i)] The distance $\displaystyle d(\Gamma,\Gamma_h):=\sup_{x\in\Gamma}\inf_{x_h\in\Gamma_h}|x-x_h|$
        between the boundaries $\Gamma$ and $\Gamma_h$ converges to zero at $d(\Gamma,\Gamma_h)\leq Ch^2$.
\item[(ii)] The unit outward normal vector $n(\phi_h^2(x))$ for $\phi_h^2(x)$ to $\Gamma_h$ converges to that at $x$ to the boundary $\Gamma$ at $|n(x)-n(\phi_h^2(x))|\leq Ch$.
\item[(iii)] The measure $|\Gamma_h|$ converges to $|\Gamma|$, i.e., $\left||\Gamma|-|\Gamma_h|\right|\leq Ch$.
\end{itemize}
\end{lem}
\begin{proof}
We discuss only the three-dimensional case, since the treatment of the two-dimensional case is
straightforward (cf. \cite{BrambleKing:1994}). By the $C^2$ regularity of the boundary
$\Gamma$ from Assumption \ref{ass:poly}, the unit outward normal vector
$n:\Gamma\mapsto \mathbb{S}^2$ is Lipschitz continuous with a Lipschitz constant $L$,
i.e., $|n(x) - n(y)| \leq L |x-y|.$

Now for any fixed $x\in\Gamma$, $d(x,\Gamma_h)$ in assertions (i) and (ii) do not depend
on the choice of the coordinate system,
and hence we may assume $n(x) = (0,0,1)^\mathrm{t}$. The implicit function theorem ensures the existence
of a neighborhood $\mathcal{N}_x\subset \Gamma$ of $x$ such that $x_3 = f(x_1,x_2)$ in
$\mathcal{N}_x$ for some $f\in C^2$. In particular, we many choose the set
$\mathcal{N}_x = \{z\in\Gamma: |z - x| < L^{-1}\}$. To see this, it suffices to show that
any line parallel to the normal vector $n(x)=(0,0,1)^\mathrm{t}$ intersects $\mathcal{N}_x$
at most once, which we prove by contradiction. Suppose that the boundary $\Gamma$ is defined by $F(x)=0$. The definition
of $\mathcal{N}_x$ implies that for any $z\in\mathcal{N}_x$, $|n(z) - n(x)|\leq 1$, i.e.,
$n(z)\cdot n(x) \geq \frac{1}{2}$ and hence $\frac{\partial F}{\partial x_3}|_{z} >0$.
Now assume the contrary, i.e., there are two points $(x_1,x_2,\tilde{x}_3)$ and $(x_1,x_2,\hat{x}_3)$ in
$\mathcal{N}_x$ with $\tilde{x}_3 < \hat{x}_3$. The tangent plane at point $\tilde{x}:=(x_1,x_2,\tilde{x}_3)$ is given by
\begin{equation*}
  \left\{z=(z_1,z_2,z_3): (z - \tilde{x})\cdot\nabla F(\tilde{x}) = 0\right\}.
\end{equation*}
By the choice of the outward normal $n(\tilde{x})$ and the convexity of the domain $\Omega$, the surface $\Gamma$ lies
below the tangent plane, i.e., $(z -\tilde{x})\cdot\nabla F \leq 0$ for all $z\in\Gamma$, which contradicts
the strictly reverse inequality for the point $(x_1,x_2,\hat{x}_3)$. This shows the desired assertion on the choice of the
neighborhood $\mathcal{N}_x$.

Note that for any $z\in \mathcal{N}_x$, we have $n(z)\cdot n(x) > \frac{1}{2}$
and hence $f_1^2(z) + f_2^2(z) < 1$. Hence there holds
$|x-z|^2 = |x_1 - z_1|^2 + |x_2 - z_2|^2 + |f(x_1,x_2) - f(z_1,z_2)|^2$ and by
the intermediate theorem, $f(x_1,x_2) -
f(z_1,z_2) = (x_1 - z_1) f_1(\xi_1) + (x_2 - z_2)f_2(\xi_2)$, where the point $(\xi_1,\xi_2)$ lies on
the line  segment from $(x_1,x_2)$ to $(z_1,z_2)$. Consequently, $|f(x_1,x_2) - f(z_1,z_2)|^2
\leq 2(|x_1 - z_1|^2 + |x_2 - z_2|^2  )$. Moreover $|x-z|^2 \leq 3 (|x_1 - z_1|^2 + |x_2 - z_2|^2 )$,
and we can conclude that the set $\{(y_1,y_2): |x_1-y_1|^2 + |x_2 -  y_2|^2 \leq \frac{1}{3L^2}\}$ is a subset
of the projection of the set $\mathcal{N}_x$ to the $x_1$-$x_2$ plane.

Now suppose that $h$ is sufficiently small. Consider the projection of $x$ onto
the $x_1$-$x_2$ plane, which intersects some triangle $\triangle ABC $, with
$A, B, C \in \mathcal{N}_x$ being vertices on the polyhedral boundary $\Gamma_h$.
Therefore, there is a surface patch $\widehat{\mathcal{N}}_x\subset \mathcal{N}_x$ (with $x\in\widehat{\mathcal{N}}_x$)
and $\triangle ABC$ can be respectively represented by $(x_1,x_2,f(x_1,x_2))$ and $(x_1,x_2,f_h(x_1,x_2))$
with $f$ and $f_h$ being $C^2$ continuous and affine, respectively. We note that
by the construction in the preceding paragraph,
such a representation also exists for the neighboring elements.
Let $n(A)$ be the unit outward normal vector at the vertex $A\in\widehat{\mathcal{N}}_x$
to the surface $\Gamma$, and $n_{ABC}$ be that of the triangle $\triangle ABC$.
Then there holds
\begin{equation*}
  \begin{aligned}
     n(A)  & = (f_1^2 + f_2^2 + 1)^{-\frac{1}{2}}(-f_1, -f_2, 1)|_{A},\\
     n_{ABC} &= (f_{1,h}^2 + f_{2,h}^2 + 1)^{-\frac{1}{2}}(-f_{1,h},-f_{2,h},1)|_{ABC},
  \end{aligned}
\end{equation*}
where $f_i = \frac{\partial f}{\partial x_i}$ and $f_{i,h} = \frac{\partial f_h}{\partial x_i}|_{ABC}$,
$i=1,2$. Now the $C^2$ regularity of $f$ yields
\begin{equation*}
   \begin{aligned}
     x_{3,B} - x_{3,A} &= f(x_{1,B},x_{2,B}) - f(x_{1,A},x_{2,A})\\
     & = f_1|_A (x_{1,B}-x_{1,A}) + f_2|_A (x_{2,B}-x_{2,A}) + O(h^2).
   \end{aligned}
\end{equation*}
Consequently, the inner product $\left|n(A)\cdot\overrightarrow{BA}\right|$ can be bounded by
\begin{equation*}
  \left|n(A)\cdot\overrightarrow{BA}\right| = \left|(f_1^2 + f_2^2 + 1)^{-\frac{1}{2}}(-f_1,-f_2, 1)|_{A} \cdot[ x_{1,B}-x_{1,A} \ \
   x_{2,B}-x_{2,A} \ \ x_{3,B}-x_{3,A}]\right|\leq Ch^2.
\end{equation*}
Similarly, one can deduce $|n(A)\cdot\overrightarrow{CA}|\leq Ch^2$.
By the quasi-uniformity of the triangulation $\mathcal{T}_h$,
the angle $\angle BAC$ is strictly bounded from below by zero, $|AB|\approx h$ and $|AC| \approx h$,
and hence the inner product between $n(A)$ and any unit vector in the plane $ABC$
is of the order $O(h)$. The normal vector $n(A)$ can be expressed as $n(A) = \alpha n_{ABC} +
n_{ABC}^\perp$ with $n_{ABC}^\perp\perp n_{ABC}$ and $\alpha\in (0,1)$ (due to the choice of orientation).
Taking inner products both sides with $n_{ABC}^\perp$ yields
\begin{equation*}
  n(A)\cdot n_{ABC}^\perp = |n_{ABC}^\perp|^2
\end{equation*}
i.e., $|n_{ABC}^\perp| = O(h)$ and $\alpha = 1 - O(h)$.
Hence, $|n(A)- n_{ABC}|^2 = (1-\alpha)^2 + |n_{ABC}|^2 =  O(h^2)$, i.e., $|n(A)-n_{ABC}|\leq Ch$.
It follows immediately from this estimate that
\begin{equation*}
  |f_1(A) - f_{1,h}| \leq Ch \quad \mbox{and}\quad |f_2(A) - f_{2,h}| \leq Ch.
\end{equation*}
With these preliminaries, now we can prove the assertions.

\textit{proof of assertion (i).} For any point $\bar{x}=(x_1,x_2,f(x_1,x_2))\in
\triangle ABC$, there holds $|x_1-x_{1,A}| + |x_2-x_{2,A}| \leq Ch$, and further
\begin{equation*}
  \begin{aligned}
    f(x_1,x_2) &= f(x_{1,A},x_{2,A}) + (x_1-x_{1,A})f_1(A) + (x_2-x_{2,A})f_2(A) + O(h^2),\\
    f_h(x_1,x_2)& = f_h(x_{1,A},x_{2,A}) + (x_1-x_{1,A})f_{1,h} + (x_2-x_{2,A})f_{2,h}.
  \end{aligned}
\end{equation*}
Upon noting the identity $f(x_{1,A},x_{2,A}) = f_h(x_{1,A},x_{2,A})$, we deduce
\begin{equation*}
|f(x_1,x_2) - f_h(x_1,x_2)| \leq |x_1-x_{1,A}||f_1(A) - f_{1,h}| + |x_2 - x_{2,A}||f_2(A) - f_{2,h}| + Ch^2 \leq Ch^2.
\end{equation*}
Therefore  $d(\bar{x},\Gamma_h) \leq |f(x_1,x_2) - f_h(x_1,x_2)| \leq Ch^2$ and assertion (i) follows.

\textit{proof of assertion (ii).} It follows from the Lipschitz continuity of the unit normal
vector $n(x)$ that
\begin{equation}\label{eqn:est:normal}
   |n(x) - n_{ABC}| \leq |n(x) - n(A)| + |n(A)- n_{ABC}| \leq L|x-x_A| + Ch \leq Ch.
\end{equation}
Now we distinguish the following cases. Case (a): $\phi_h^2(x)\in \triangle ABC$, then the assertion
follows directly from \eqref{eqn:est:normal} and the fact that $n(\phi_h^2(x))=n_{ABC}$.
Case (b): $\phi_h^2(x)\notin\triangle ABC$. By part (i), $|x-\phi_h^2(x)|=d(x,\Gamma_h)\leq Ch^2$,
i.e., $\phi_h^2(x)$ lies within an $O(h^2)$ neighborhood of the point $x$. Let $\bar{x}\in
\triangle ABC$ be the intersection point defined by the projection through the point $x\in\Gamma$
onto the $x_1$-$x_2$ plane. Then by the triangle inequality, $|\bar{x}-\phi_h^2(x)|\leq Ch^2$.
Then the $C^2$ regularity of $f$ yields $|x-\check{\phi}_h^2(x)|\leq Ch^2$, where $\check{\phi}_h^2(x)$
denotes the pull-back of the orthogonal projection of $\phi_h^2(x)$ (onto the $x_1$-$x_2$ plane)
to the boundary $\Gamma$. The Lipschitz continuity of the normal vector, \eqref{eqn:est:normal}
and the triangle inequality yield
\begin{equation*}
  |n(x)-n(\phi_h^2(x))| \leq |n(x)-n(\check{\phi}_h^2(x))| + |n(\check{\phi}_h^2(x))-n(\phi_h^2(x))| \leq Ch.
\end{equation*}
This completes the proof of assertion (ii).

\textit{proof of assertion (iii)}. Here we first consider a local patch. For any subset
$\mathcal{N}\subset \mathcal{N}_x$ (with the choice $n(x)=(0,0,1)^\mathrm{t}$) with a boundary $\partial\mathcal{N}$
of finite perimeter, let $\mathcal{N}_h\subset \Gamma_h$ be the approximation of $\mathcal{N}$
which consists of triangles with all vertices lying on $\mathcal{N}$, and $\mathcal{N}_{x_1,x_2}
\subset\mathbb{R}^2$ be the orthogonal projection of $\mathcal{N}_h$ onto the $x_1$-$x_2$ plane.
According to the preceding construction, the projection is well defined.
Then for any $(x_1,x_2)\in \mathcal{N}_{x_1,x_2}$, $(x_1,x_2,f_h(x_1,x_2))\in \mathcal{N}_h$ and
$(x_1,x_2,f(x_1,x_2))\in \mathcal{N}$. Now for any triangle $\triangle ABC \subset \mathcal{N}_h$,
$f_{i,h}|_{\triangle ABC}$ is constant, and consequently
\begin{equation*}
\left|f_i|_{(x_1,x_2)} - f_{i,h}|_{(x_1,x_2)}\right| \leq \left|f_i|_{(x_1,x_2)} - f_i|_{A}\right| + \left|f_i|_{A} - f_{i,h}|_{A}\right| \leq Ch.
\end{equation*}
Let $\widetilde{\mathcal{N}} = \{(x_1,x_2,f(x_1,x_2)):(x_1,x_2)\in\mathcal{N}_{x_1,x_2}\}\subset\mathcal{N}$. Then
\begin{equation*}
  \begin{aligned}
     |\widetilde{\mathcal{N}}| = \int_{\mathcal{N}_{x_1,x_2}} \sqrt{f_1^2 + f_2^2 +1}dx_1dx_2\quad \mbox{and}\quad
     |\mathcal{N}_h| = \int_{\mathcal{N}_{x_1,x_2}} \sqrt{f_{1,h}^2 + f_{2,h}^2 +1}dx_1dx_2,
  \end{aligned}
\end{equation*}
where the derivatives $f_{i,h}$ should be understood elementwise. Therefore,
\begin{equation*}
\left||\widetilde{\mathcal{N}}| - |\mathcal{N}_h|\right| \leq \int_{\mathcal{N}_{x_1,x_2}} |f_1 - f_{1,h}| + |f_2 - f_{2,h}| \leq Ch.
\end{equation*}
Moreover, since the mesh size is $h$, and the surface $\Gamma$ is $C^2$, the set
$\mathcal{N}\backslash\widetilde{\mathcal{N}}$ is contained in the set
$\{x: d(x,\partial\mathcal{N}) \leq Ch\}$. Consequently, we have
\begin{equation*}
\left||{\mathcal{N}}| - |\mathcal{N}_h|\right| \leq \left||\widetilde{\mathcal{N}}| - |\mathcal{N}_h|\right| + \left|\mathcal{N}\backslash\widetilde{\mathcal{N}}\right| \leq Ch.
\end{equation*}
Now we estimate the approximation of the whole boundary. Since the boundary $\Gamma$ is compact
and $\Gamma\subset \cup_{x\in\Gamma}\mathcal{N}_x$, by the finite opening cover
theorem, there exists a finite number of points $\{x^i\}_{i=1}^n$ such that
$\Gamma\subset \cup_{x^i}\mathcal{N}_{x^i}$. Let
$\mathcal{N}^1 = \mathcal{N}_{x^1}$, $\mathcal{N}^2 = \mathcal{N}_{x^2}\backslash\mathcal{N}^1$,
$\mathcal{N}^i = \mathcal{N}_{x^i}\backslash\cup_{j=1}^{i-1}\mathcal{N}^j$, $i=3,...,n$.
Then each set $\mathcal{N}_{x^i}$ has a boundary of finite measure.
Clearly there holds $|\Gamma| = \sum_i|\mathcal{N}^i|$ and $|\Gamma_h|\geq \sum_i |\mathcal{N}^i_h|.$
Now by the nonexpansiveness of the map $\phi_h^2$, there holds
$|\Gamma_h| \leq |\Gamma|$. Now assertion (iii) follows from
\begin{equation*}
  |\Gamma| - |\Gamma_h| \leq \sum_i \left||\mathcal{N}^i| - |\mathcal{N}^i_h|\right| \leq Ch.
\end{equation*}
This completes the proof of the lemma.
\end{proof}

\begin{rmk}
The quasi-uniformity of the mesh is essential for Lemma \ref{lem:assump}. One can find a counterexample on the approximation
of surface area in \cite[Section 623]{Fihtengolc:2003} in the absence of the quasi-uniformity condition;
see also \cite[Sections 624 and 627]{Fihtengolc:2003} for related discussions.
\end{rmk}

With the help of Lemma \ref{lem:assump}, we can show the following properties of the maps $\phi_h^1$, $\phi_h^2$
and $\varrho_h$ defined at the beginning of this section, which are crucial for the convergent analysis below.
\begin{lem}\label{ass:map}
Let Assumption \ref{ass:poly} be fulfilled. Then there exists a function $\epsilon_h\rightarrow 0$ as
$h \rightarrow 0$ such that the maps $\phi_h^1$, $\phi_h^2$ and $\varrho_h$ satisfy:
  \begin{itemize}
    \item[(i)] For any element $S_h$, there holds $\tfrac{|(\phi_h^1)^{-1}(S_h)|}{|T_h|} \leq \epsilon_h$.
    \item[(ii)] For any subset $e\subset\Gamma$, there holds
     $|\phi_h^2(e)| \leq |e|$, $|\phi_h^2(e)|\rightarrow |e|$ and $|\phi_h^2(e)| \leq |\varrho_h(e)| \leq (1+\epsilon_h)|\phi_h^2(e)|$.
  \end{itemize}
\end{lem}
\begin{proof}
By definition, there holds $|\phi_h^1(x) - x| \leq d(\Gamma,\Gamma_h)$,
and $(\phi_h^1)^{-1}(S_h)\subset \{x\in\Omega\setminus\Omega_h: d(x,S_h) \leq d(\Gamma,\Gamma_h)\}$.
Hence, for any $S_h$, the measure of the set $\{x\in \Omega\setminus\Omega_h:d(x,S_h)
\leq d(\Gamma,\Gamma_h)\}$ satisfies $\left|\{x\in \Omega\setminus\Omega_h:d(x,S_h)
\leq d(\Gamma,\Gamma_h)\}\right| \leq Ch^{d+1}$, in view of Lemma \ref{lem:assump}(i)
and the fact $|S_h|\approx h^{d-1}$, and the quasi-uniformity of
$\mathcal{T}_h$ implies $|T_h| \approx h^d$, from which assertion (i) follows directly.

Since the map $\phi_h^2$ is nonexpansive, $|\phi_h^2(e)| \leq |e|$ and $|\phi_h^2(e)| = |\phi_h^2(\varrho_h(e))| \leq |\varrho_h(e)|$.
Meanwhile, clearly there holds $|e| - |\phi_h^2(e)| \leq |\Gamma| - |\Gamma_h|$, and consequently
$|\phi_h^2(e)| \rightarrow |e|$ by Lemma \ref{lem:assump}(iii). Now by Lemma \ref{lem:assump}(ii)
we have $\delta_h = \sup_{x}|n_x - n_{\phi_h^2(x)}| \leq Ch$. 
Without loss of generality, we can assume that $\phi_h^2(e) \subset \mbox{int} S_h$.
By using a local coordinate system, let $S_h$ be in the $x_1$-$x_2$ plane ($x_1$ axis in the two-dimensional case).
Then the set $\varrho_h(e)\subset \psi_h(\mbox{int}S_h)$ can be represented by a nonnegative function $f(x_1,x_2)$ for
$(x_1,x_2)\in \mathrm{int}S_h$. It follows from the definition of $\delta_h$ that
\begin{equation*}
\left|\tfrac{(f_1,f_2,1)}{\sqrt{1+f_1^2 + f_2^2}} - (0,0,1)\right|\leq \delta_h,
\end{equation*}
from which it follows that $f_1^2 + f_2^2 \leq \tfrac{\delta^2_h}{1-\delta_h^2}$.
However, the area for the patch $\varrho_h(e)$ is given by
\begin{equation*}
|\varrho_h(e)| = \int_{\phi_h^2(e)} \sqrt{1+f_1^2 + f_2^2} dxdy \leq \left(1+ \tfrac{\delta_h}{\sqrt{1-\delta_h^2}}\right)|\phi_h^2(e)|.
\end{equation*}
This shows the second assertion. Then $\epsilon_h$ can be properly chosen to satisfy both (i) and (ii).
\end{proof}

We shall also need the following lemma.
\begin{lem}\label{lem:reg1}
Let $p\geq 1$ and the domain $\Omega$ divide into $n$ disjoint open subdomains $\{\Omega_i\}$, with a Lipschitz
interface between every neighboring subdomains. Then
$u|_{\Omega_i}\in W^{1,p} (\Omega_i)$ and $u\in C(\overline{\Omega})$ imply $u\in W^{1,p}(\Omega)$.
\end{lem}
\begin{proof} It suffices to consider the case $n=2$, i.e., two domains $\Omega_1$ and $\Omega_2$ with
the interface being $\Gamma_{1,2}$. We define functions $v_j$, $j=1,...,d$, by $v_j(x) =
\frac{\partial u}{\partial x_j} \chi_{\Omega_1\cup\Omega_2}$. Then for any $\phi\in C_0^\infty(\Omega)$, we have
\begin{equation*}
\int_{\Omega}v_j\phi dx = \int_{\Omega_1} v_j\phi dx+ \int_{\Omega_2} v_j\phi dx= -\int_{\Omega_1} u\phi_{x_j}dx - \int_{\Omega_2} u\phi_{x_j}dx + \int_{\Gamma_{1,2}}[u\phi] n_{x_j}ds,
\end{equation*}
where $[\cdot]$ denotes the jump across the interface $\Gamma_{1,2}$
and $n_{x_j}$ is the $j$th component of the unit outward normal vector to the boundary
$\partial\Omega_1$. By the continuity of $u$, the jump term on the interface
$\Gamma_{1,2}$ vanishes identically, and thus
$\int_{\Omega}v_j\phi dx= - \int_{\Omega} u\phi_{x_j}dx,$
i.e., $v_j$ is the weak derivative of $u$. Clearly, the function $v_j$ belongs to
the space $L^p(\Omega)$. This concludes the proof of the lemma.
\end{proof}

The finite element solution is only defined on the domain $\Omega_h$, whereas
the true solution is defined on the domain $\Omega$. In order to compare them,
we introduce an extension operator $\jmath:\Omega_h\mapsto\Omega$ as follows:
\begin{equation}\label{equ:extension}
\jmath {v}_h(x) = \left\{\begin{aligned}
v_h(x), &\quad x\in\overline{\Omega}_h, \\
v_h(\phi_h(x)), &\quad x\in \overline{\Omega} \backslash \overline{\Omega}_h, \end{aligned}\right.
\forall v_h\in V_h.
\end{equation}

The extension operator $\jmath$ satisfies the following estimate.
\begin{lem}\label{lem:femext}
Let $p\geq1$ and Assumption \ref{ass:poly} be fulfilled. Then there holds
\begin{equation*}
   \|\jmath v_h\|_{W^{1,p}(\Omega\backslash\Omega_h)}\leq C\epsilon_h^\frac{1}{p} \|v_h\|_{W^{1,p}(\Omega_h)}\quad\forall v_h\in V_h.
\end{equation*}
\end{lem}
\begin{proof} Clearly, $\jmath v_h\in C(\overline{\Omega})$ and the projection
operator $\phi_h$ is nonexpansive, i.e., $|\phi_h(x)-\phi_h(y)| \leq |x-y|$. Thus for
any $x\in\Omega\backslash\overline{\Omega}_h$ with $\phi_h(x) \in S_h\subset T_h$,
$T_h\in\mathcal{T}_h$, there holds
\begin{equation*}
  \begin{aligned}
     |\nabla \jmath v_h(x)| &= \limsup_{y\rightarrow x}\frac{|\jmath v_h(x) - \jmath v_h(y) | }{|x-y|} \\
      & \leq \limsup_{\substack{y\rightarrow x\\
         \phi_h(y)\neq\phi_h(x)}}\frac{|v_h(\phi_h(x)) - v_h(\phi_h(y))|}{|\phi_h(x) - \phi_h(y)|} \leq \|\nabla v_h\|_{L^\infty(S_h)},
  \end{aligned}
\end{equation*}
and
\begin{equation*}
|\jmath v_h(x)| = |v_h(\phi_h(x))| \leq \|v_h\|_{L^\infty(S_h)}.
\end{equation*}
Further, $\|\nabla v_h\|_{L^\infty(S_h)}\leq \|\nabla v_h\|_{L^\infty(T_h)}$ and
$\|v_h\|_{L^\infty(S_h)}\leq \|v_h\|_{L^\infty(T_h)}$, and thus,
\begin{equation}\label{ine:extvh}
\begin{aligned}\|\nabla \jmath v_h\|_{L^p((\phi_h^1)^{-1}(S_h))}^p 
  &\leq |(\phi_h^1)^{-1}(S_h)|\|\nabla v_h\|^p_{L^\infty(T_h)},\\
  \|\jmath v_h\|_{L^p((\phi_h^1)^{-1}(S_h))}^p 
  &\leq |(\phi_h^1)^{-1}(S_h)|\|v_h\|^p_{L^\infty(T_h)}.
\end{aligned}
\end{equation}
Now for the element $T_h\in\mathcal{T}_h$, we consider an affine transformation
$\mathcal{F}:\widehat{T} \mapsto T_h$, $\mathcal{F}(\hat{x}) = J \widehat{x} + b$,
where $\widehat{T}$ is the reference element. The quasi-uniformity of the
triangulation $\mathcal{T}_h$ implies \cite{Ciarlet:2002}
\begin{eqnarray}\label{ine:J}
|\mbox{det}(J)| = |T_h|/|\widehat{T}|
\approx h^d, \quad \|J\| \approx h, \quad \|J^{-1}\| \approx
h^{-1}.
\end{eqnarray}
where $\approx$ means being of the same order, and $\|\cdot\|$ is the matrix spectral norm.
Then by a change of variable, chain rule
and \eqref{ine:J}, we deduce that for any $s\geq0$
\begin{equation*}
|\widehat{v}|_{W^{s,p}(\widehat{T})} \approx \left\{ \begin{aligned} h^{s-\frac{d}{p}}
|v|_{W^{s,p}(T_h)}, &\quad 1\leq p <\infty,\\
h^{s} |v|_{W^{s,p}(T_h)}, &\quad p= \infty.
\end{aligned}\right.
\end{equation*}
Consequently, we have
\begin{equation*}
  \begin{aligned}
     |v_h|_{W^{1,\infty}(T_h)}^p & \approx h^{-p}|\widehat{v}_h|_{W^{1,\infty}(\widehat{T})}^p \approx h^{-p}|\widehat{v}_h|_{W^{1,p}(\widehat{T})}^p \\
     &\approx h^{-d}|v_h|_{W^{1,p}(T_h)}^p \approx \tfrac{1}{|T_h|}|v_h|_{W^{1,p}(T_h)}^p,\\
     |v_h|_{L^\infty(T_h)}^p &\approx |\widehat{v}_h|_{L^\infty(\widehat{T})}^p \approx |\widehat{v}_h|_{L^p(\widehat{T})}^p \\
     &\approx h^{-d}|v_h|_{L^p(T_h)}^p \approx \tfrac{1}{|T_h|}|v_h|_{L^p(T_h)}^p.
  \end{aligned}
\end{equation*}
This together with \eqref{ine:extvh} yields
\begin{eqnarray*}
&&\|\nabla \jmath v_h\|_{L^p((\phi_h^1)^{-1}(S_h))}^p \leq C\frac{|(\phi_h^1)^{-1}(S_h)|}{|T_h|}\|\nabla v_h\|_{L^p(T_h)}^p,\\
&&\| \jmath v_h\|_{L^p((\phi_h^1)^{-1}(S_h))}^p \leq C\frac{|(\phi_h^1)^{-1}(S_h)|}{|T_h|}\| v_h\|_{L^p(T_h)}^p .
\end{eqnarray*}
Now summing over all $S_h$ and part (i) of Lemma \ref{ass:map} yield the desired assertion.
\end{proof}

The next result estimates the error of the boundary term.
\begin{lem}\label{lem:femextbry}
Let Assumption \ref{ass:poly} be fulfilled. Then for any $v_h\in {V}_h$, a subset $e_h\subset \mathrm{int}S_h$,
and $\tilde{e}_h = \psi_h(e_h)$, there holds
\begin{equation*}
\left|\int_{e_h} v_hds - \int_{\tilde{e}_h} \jmath v_hds \right| \leq \epsilon_h \|v_h\|_{L^1(e_h)}.
\end{equation*}
\end{lem}
\begin{proof} By the continuity of the functions $v_h$ and $\jmath v_h$,
Riemann and Lebesgue integration coincides. We show the assertion using the definition of Riemann integration.
Let $\cup_i p_{i}$ be a partition of $\tilde{e}_h$. Then $\cup_i \phi_h^2(p_i)$ forms a partition of $e_h$.
Since the map $\phi_h^2$ is nonexpansive, the norm of the partition $\{\phi_h^2(p_i)\}$ tends to zero as
that of $\{p_i\}$ goes to zero. For any choice of $\{x_i\in p_i\}$, there holds
\begin{equation*}
  \begin{aligned}
    \left|\int_{e_h} v_hds - \int_{\tilde{e}_h} \jmath v_hds\right| &= \lim_{n\rightarrow \infty} \left|\sum_{i} (v_h(\phi_h^2(x_i))|\phi_h^2(p_i)| - \jmath v_h(x_i)|p_i|)\right| \\
    &= \lim_{n\rightarrow \infty} \left|\sum_{i} v_h(\phi_h^2(x_i))(|\phi_h^2(p_i)| -|p_i|)\right| .
  \end{aligned}
\end{equation*}
By part (ii) of Lemma \ref{ass:map}, there holds $\left||\phi_h^2(p_i)| - |p_i|\right|\leq \epsilon_h
|\phi_h^2(p_i)|$, which concludes the proof.
\end{proof}
\begin{rmk}
An inspection of the proof indicates that the lemma is valid for any function continuous over the domain $\overline{\Omega}_h$.
\end{rmk}

We shall also need a Riesz projection $\mathcal{R}_h$, which is dependent of the domain $\Omega_h$.
\begin{lem}\label{lem:riesz}
Let the Riesz projection $\mathcal{R}_h: H^1(\Omega)\mapsto V_h$ be defined by
\begin{equation*}
  \int_{\Omega_h} \nabla \mathcal{R}_hv\cdot\nabla v_hdx +\int_{\Omega_h} \mathcal{R}_hvv_hdx
  = \int_{\Omega_h} \nabla v\cdot\nabla v_hdx + \int_{\Omega_h} vv_hdx\quad \forall v_h\in V_h.
\end{equation*}
Then the operator $\mathcal{R}_h$ satisfies the following estimate
\begin{equation*}
  \lim_{h\to0}\|\jmath\mathcal{R}_hv-v\|_{H^1(\Omega)}=0\quad \forall v\in H^1(\Omega).
\end{equation*}
\end{lem}
\begin{proof}
Clearly, by the definition of the space $H^1(\Omega)$, there holds
\begin{equation*}
\|\jmath\mathcal{R}_hv-v\|^2_{H^1(\Omega)} = \|\mathcal{R}_hv-v\|^2_{H^1(\Omega_h)} + \|\jmath\mathcal{R}_hv-v\|^2_{H^1(\Omega\backslash\Omega_h)}.
\end{equation*}
It suffices to estimate the two terms. We estimate the first term by a density argument.
By C\'{e}a's lemma and the definition of Riesz projection, $\|\mathcal{R}_hv-v\|_{H^1(\Omega_h)}\leq
\inf_{v_h\in V_h}\|v_h-v\|_{H^1(\Omega_h)}$. Then following
\cite[Theorem 3.2.3]{Ciarlet:2002} we deduce that for any
$v\in C^\infty(\overline{\Omega})$, there holds
\begin{equation*}
\|v - \mathcal{I}_h v\|_{H^1(\Omega_h)} \leq C h |\Omega_h|^{\frac{1}{2}} \|v\|_{W^{2,\infty}(\Omega_h)},
\end{equation*}
where $C$ does not depend on $\Omega_h$.
Now for any fixed $v\in H^1(\Omega)$, by the density of $C^\infty(\overline{\Omega})$ in $H^1(\Omega)$,
there exists $v^\epsilon\in C^\infty(\overline{\Omega})$ with $\|v^\epsilon - v\|_{H^1(\Omega)}\leq
\epsilon$ for any $\epsilon>0$. Hence with the choice $v_h=\mathcal{I}_h v^\epsilon$ in C\'{e}a's lemma, there holds
$\lim_{h\to0}\|\mathcal{R}_h v - v\|_{H^1(\Omega_h)} = 0$.
Meanwhile, by the triangle inequality we have
\begin{equation*}
   \begin{aligned}
    \|\jmath\mathcal{R}_hv-v\|_{H^1(\Omega\backslash\Omega_h)} \leq& \|\jmath\mathcal{R}_hv\|_{H^1(\Omega\backslash\Omega_h)} + \|v\|_{H^1(\Omega\backslash\Omega_h)}\\
    \leq & C\epsilon_h^\frac{1}{2}\|\mathcal{R}_hv\|_{H^1(\Omega_h)} + \|v\|_{H^1(\Omega\backslash\Omega_h)} \rightarrow 0,
   \end{aligned}
\end{equation*}
where the last term tends to zero by Lebesgue dominated convergence theorem \cite{Evans:1992a}.
\end{proof}

Next we establish an analogue of Lemma \ref{lem:polyhed} for
the solution $(u_h,U_{h})\in \mathbb{H}_h:=H^1(\Omega_h)\otimes\mathbb{R}_\diamond^L$
to the discrete variational problem (on the polyhedral domain $\Omega_h$):
\begin{equation}\label{equ:dis}
 \int_{\Omega_h}\sigma_h\nabla u_h\cdot\nabla v_h dx+\sum_{l=1}^L z_l^{-1}\int_{e_{l,h}}
 (u_h-U_{h,l})(v_h-V_l)ds =\sum_{l=1}^LI_lV_l, \quad \forall (v_h,V)\in \mathbb{H}_h.
\end{equation}
There are several possible choices of the discrete surface $e_{h,l}$, which is a polyhedral approximation
to the surface patch $e_{l}$ (occupied by the electrode). A straightforward definition of $e_{h,l}$ would
be $e_{l,h}=\phi_h(e_l)$. Here we let $e_{h,l}=\phi_h(\varrho_h(e_l))$, with $\varrho_h(e_l)$ given in \eqref{equ:e1}.
We note that in practice, the surface $e_{h,l}$ can be chosen to be the union of a collection of polyhedral
surfaces only so as to avoid integration over a curved surface; and the analysis below remains valid for this case.
Next we denote $\tilde{e}_{l,h}=\psi_h(e_{l,h})\subset e_l$. By Assumption \ref{ass:poly}(b), the measure
$|e_l\backslash{\tilde{e}_{l,h}}|$ is bounded by $|\Gamma|-|\Gamma_h|$, with a limit zero as $h$ goes to
zero. By the Lax-Milgram theorem, for each fixed $\sigma_h$, there exists a unique solution $(u_h,U_h)\in\mathbb{H}_h$ to
the discrete variational problem \eqref{equ:dis}, and it satisfies the following a priori error estimate
$\|(u_h,U_h)\|_{\mathbb{H}_h}\leq C\|I\|$. Hence by Lemma \ref{lem:femext}, the sequence $\{(\jmath u_h,
U_h)\}$ is uniformly bounded in $\mathbb{H}$ independent of $h$.
Now we can state an important lemma on the convergence of the discrete forward map
$\sigma_h\mapsto (\jmath u_h(\sigma_h),U_h(\sigma_h))\in \mathbb{H}_h$.
\begin{lem}\label{lem:discon}
Let $\{\sigma_h\}\subset \mathcal{A}$ and $(u_h(\sigma_h), U_{h}(\sigma_h))$ solve the discrete
variational problem \eqref{equ:dis}. If $\jmath\sigma_h$ converges to $\sigma\in \mathcal{A}$ in
$L^1(\Omega)$, then the sequence $\{(\jmath u_h(\sigma_h),U_h(\sigma_h))\}$ converges to
$(u(\sigma),U(\sigma))$ in $\mathbb{H}$.
\end{lem}
\begin{proof}
For simplicity, we denote the extensions of $\sigma_h$, $u_h$ and $v_h$ from $\Omega_h$ to $\Omega$ by
$\tilde{\sigma}_h = \jmath\sigma_h$, $\tilde{u}_h = \jmath u_h(\sigma_h)$, $\tilde{v}_h = \jmath v_h$, and $\bar{u}_h=\jmath\mathcal{R}_hu$.
Then the assumption $\sigma_h \in\mathcal{A}_h\subset V_h$, i.e., $c_0\leq \sigma_h\leq c_1$,
and a priori estimate for $u_h$, Lemmas \ref{lem:femext} and
\ref{lem:reg1}, imply that the sequence $\{\tilde{u}_h(\sigma_h)\}$ is
uniformly bounded in $H^1(\Omega)$.

First we rewrite the discrete variational formulation \eqref{equ:dis} as
\begin{equation*}
  \begin{aligned}
   \int_{\Omega}\tilde{\sigma}_h\nabla \tilde{u}_h\cdot\nabla \tilde{v}_h dx+&\sum_{l=1}^L z_l^{-1}\int_{e_l}(\tilde{u}_h-U_{h,l})(\tilde{v}_h-V_l)ds
     =\int_{\Omega\setminus\Omega_h}\tilde{\sigma}_h\nabla \tilde{u}_h\cdot\nabla \tilde{v}_h dx + \sum_{l=1}^LI_lV_l\\
     &+\sum_{l=1}^L z_l^{-1}\int_{e_l}(\tilde{u}_h-U_{h,l})(\tilde{v}_h-V_l)ds - \sum_{l=1}^L z_l^{-1}\int_{e_{l,h}}(u_h-U_{h,l})(v_h-V_l)ds, \quad \forall (v_h,V)\in \mathbb{H}_h.
  \end{aligned}
\end{equation*}
and take the test function $(v_h,V)=(\mathcal{R}_hu-u_h,U-U_h)$. Next we subtract it from \eqref{eqn:cemweakform}
with the test function $(v,V)=(\bar{u}_h-\tilde{u}_h,U-U_h)$ to get the identity for the error $(w,W)=(u-\bar{u}_h,U-U_h)$:
\begin{equation*}
   \begin{aligned}
       \int_\Omega \tilde{\sigma}_h|\nabla w|^2dx &+ \sum_{l=1}^Lz_l^{-1}\int_{e_l}|w-W_l|^2ds
     \leq \underbrace{\int_\Omega(\tilde{\sigma}_h-\sigma)\nabla u \cdot\nabla (\bar{u}_h-\tilde{u}_h)dx}_{I} + \underbrace{\int_\Omega\tilde{\sigma}_h\nabla w\cdot\nabla (u-\bar{u}_h)dx}_{II} \\& + \underbrace{\int_{\Omega\setminus\Omega_h}\tilde{\sigma}_h\nabla \tilde{u}_h\cdot\nabla (\bar{u}_h-\tilde{u}_h)dx}_{III}
      +\sum_{l=1}^Lz_l^{-1}\underbrace{\int_{e_l}(w-W_l)(u-\bar{u}_h)ds}_{IV}\\
      &+ \sum_{l=1}^L z_l^{-1}\underbrace{\left[\int_{e_l}(\tilde{u}_h-U_{h,l})((\bar{u}_h-\tilde{u}_h)-W_l)ds - \int_{e_{l,h}}(u_h-U_{h,l})((\mathcal{R}_hu-u_h)-W_l)ds\right]}_{V}.
   \end{aligned}
\end{equation*}

Next we estimate the five terms ($I$--$V$) on the right hand side. For the first term $I$, by the
generalized H\"{o}lder's inequality we have
\begin{equation*}
  |I|\leq \|\tilde{\sigma}_h-\sigma\|_{L^p(\Omega)}\|\nabla u\|_{L^q(\Omega)}\|\nabla(\bar{u}_h-\tilde{u}_h)\|_{L^2(\Omega)},
\end{equation*}
where the exponent $q$ is from Theorem \ref{thm:cemreg}, and the exponent $p>0$ satisfies
$\frac{1}{p}+\frac{1}{q}=\frac{1}{2}$. The factor $\|\nabla (\bar{u}-\tilde{u}_h)\|_{L^2(\Omega)}$
is uniformly bounded due to the bounds on $u_h$ and $\mathcal{R}_hu$
and Lemma \ref{lem:femext}. Meanwhile, there holds $\|\tilde{\sigma}_h-\sigma\|_{L^p(\Omega)}\leq
C\|\tilde{\sigma}_h-\sigma\|_{L^1(\Omega)}^\frac{1}{p}\to0$. Hence, the first term $I\to0$ as $h\to0$.
For the second term $II$, in view of Lemma \ref{lem:riesz} and the uniform bound of the
discrete admissible set $\mathcal{A}_h$, we have
\begin{equation*}
    |II| \leq \|\tilde{\sigma}_h\|_{L^\infty(\Omega)}\|\nabla w\|_{L^2(\Omega)}\|\nabla (u-\bar{u}_h)\|_{L^2(\Omega)}\to0.
\end{equation*}
Similarly, Lemma \ref{lem:femext} and uniform boundedness of $\|\nabla u_h\|_{L^2(\Omega)}$
and $\|\nabla (\mathcal{R}_hu-u_h)\|_{L^2(\Omega)}$ yield
\begin{equation*}
  \begin{aligned}
    |III|&\leq C\|\tilde{\sigma}_h\|_{L^\infty(\Omega\backslash\Omega_h)}\|\nabla\tilde{u}_h\|_{L^2(\Omega\backslash\Omega_h)}
      \|\nabla(\bar{u}_h-\tilde{u}_h)\|_{L^2(\Omega\backslash\Omega_h)} \\
    &\leq C\epsilon_h\|\nabla u_h\|_{L^2(\Omega_h)}\|\nabla(\mathcal{R}_hu-u_h)\|_{L^2(\Omega_h)} \rightarrow 0.
  \end{aligned}
\end{equation*}

Next we consider the boundary terms. By the trace theorem \cite{Evans:1992a} and the approximation
property of $\mathcal{R}_h$ in Lemma \ref{lem:riesz}, we get
\begin{equation*}
  \begin{aligned}
   |IV| & \leq\|w-W_l\|_{L^2(e_l)}\|u-\bar{u}_h\|_{L^2(e_l)}\leq C\|w-W_l\|_{L^2(e_l)}\|u-\bar{u}_h\|_{H^1(\Omega)}\to0.
  \end{aligned}
\end{equation*}
Lastly, for the term $V$, it suffices to consider the quantity
$\int_{e_{l,h}}u_h v_hds-\int_{e_l}\tilde{u}_h\tilde{v}_hds$ with $v_h=\mathcal{R}_hu-u_h$,
and the remaining terms can be bounded similarly.
By letting $\tilde{e}_{l,h} = \psi_h(e_{l,h})$, the triangle inequality and Lemma \ref{lem:femextbry}, we deduce that
\begin{equation*}
  \begin{aligned}
    \left|\int_{e_{l,h}}\! u_h v_hds -\int_{e_l}\! \tilde{u}_h \tilde{v}_hds \right|
    \leq&\left|\int_{e_{l,h}} u_h v_hds - \int_{\tilde{e}_{l,h}} \tilde{u}_h \tilde{v}_hds\right| + \left|\int_{e_l\backslash{\tilde{e}_{l,h}}}
    \tilde{u}_h\tilde{v}_hds \right|\\
    \leq& \epsilon_h\|u_h v_h\|_{L^1(e_{l,h})} + \|\tilde{u}_h\tilde{v}_h\|_{L^1(e_l\backslash{\tilde{e}_{l,h}})} \\
    \leq&\epsilon_h\|u_h v_h\|_{L^1(e_{l,h})} + \|\tilde{u}_h\tilde{v}_h\|_{L^1(e_l\backslash{\tilde{e}_{l,h}})}
     =: I_1 + I_2.
  \end{aligned}
\end{equation*}
It remains to bound the terms $I_1$ and $I_2$. By Sobolev embedding
theorem \cite{Evans:1992a} and uniform boundness of $u_h$ in $H^1(\Omega_h)$,
$\|u_h v_h\|_{L^1(e_{l,h})}$ is uniformly bounded for all $h$,
hence $I_1\rightarrow 0$ as $h\rightarrow 0$. By Sobolev embedding
theorem, $H^1(\Omega)$ embeds continuously into $L^4(\Gamma)$ ($d=2,3$), and
thus by H\"{o}lder's inequality, we deduce $\tilde{u}_h\tilde{v}_h\in L^2(\Gamma)$.
Hence we can estimate the term $I_2$ by H\"{o}lder's inequality, the trace theorem
and the uniform boundedness of $\tilde{u}_h$ and $\tilde{v}_h$ in $H^1(\Omega)$ as follows
\begin{equation*}
  \begin{aligned}
    I_2 &\leq \|\tilde{u}_h\tilde{v}_h\|_{L^2(e_l\setminus\tilde{e}_{l,h})}|e_l\setminus\tilde{e}_{l,h}|^\frac{1}{2}\\
      & \leq \|\tilde{u}_h\|_{L^4(e_l)}\|\tilde{v}_h\|_{L^4(e_l)}|e_l\setminus\tilde{e}_{l,h}|^\frac{1}{2}\\
      & \leq C\|\tilde{u}_h\|_{H^1(\Omega)}\|\tilde{v}_h\|_{H^1(\Omega)}|e_l\setminus\tilde{e}_{l,h}|^\frac{1}{2}\to 0.
  \end{aligned}
\end{equation*}
Now Lemma \ref{lem:discon} follows directly from the preceding estimates.
\end{proof}

Finally, we analyze the discrete optimization problem for curved domains:
\begin{equation}\label{eqn:disopt:curved}
 \min_{\sigma_h\in\mathcal{A}_h}\left\{J_h(\sigma_h) = \tfrac{1}{2}\|U_h(\sigma_h)-U^\delta\|^2 + \eta\Psi_h(\sigma_h)\right\},
\end{equation}
where the discrete approximation $U_h(\sigma_h)$ is defined by the finite element
system \eqref{equ:dis}, and the discrete penalty functional $\Psi_h(\sigma_h)$ is defined by
\begin{equation*}
  \Psi_h(\sigma_h) = \left\{\begin{aligned}
    \tfrac{1}{2}\|\jmath\sigma_h\|_{H^1(\Omega_h)}^2, & \mbox{ smoothness},\\
    |\jmath\sigma_h|_{\mathrm{TV}(\Omega_h)}, & \mbox{ total variation}.
  \end{aligned}\right.
\end{equation*}
We observe that the penalty functional is defined only on the polyhedral approximation
$\Omega_h$, so the discrete optimization problem involves only computations on the approximate domain
$\Omega_h$ as well. Like before, the existence of a minimizer $\sigma_h^\ast\in\mathcal{A}_h$
to the discrete functional $J_h(\sigma_h)$ follows immediately from the compactness and
norm equivalence in finite-dimensional spaces.

We now can show the convergence of the finite element approximation for curved domains.
\begin{thm}
Let Assumption \ref{ass:poly} be fulfilled and $\{\sigma_h^\ast\}_{h>0}$ be a sequence
of minimizers to problem \eqref{eqn:disopt:curved}.
Then the sequence $\{\jmath\sigma_h^\ast\}_{h>0}$ contains a convergent subsequence to a
minimizer of problem \eqref{eqn:tikh} as
the mesh size $h$ tends to zero.
\begin{itemize}
  \item[(a)] The convergence is weakly in $H^1(\Omega)$, if $\Psi_h(\sigma_h)=\tfrac{1}{2}\|\sigma_h\|_{H^1(\Omega_h)}^2$;
  \item[(b)] The convergence is in $L^1(\Omega)$, if $\Psi_h(\sigma_h)=|\sigma_h|_{\mathrm{TV}(\Omega_h)}$.
\end{itemize}
\end{thm}
\begin{proof}
We note that the constant function $\sigma_h=1$ (with $\jmath\sigma_h=1$) lies in the
admissible set $\mathcal{A}_h$ for any $h$. Therefore, the sequence
$\{\Psi_h(\sigma_h^\ast)\}$ is uniformly bounded. Next
\begin{equation*}
  \Psi(\jmath\sigma_h^\ast)-\Psi_h(\sigma_h^*) = \left\{\begin{aligned}
    \tfrac{1}{2}\|\jmath\sigma_h^\ast\|_{H^1(\Omega\setminus\Omega_h)}^2, & \mbox{ case (a)},\\
    |\jmath\sigma_h^\ast|_{\mathrm{TV}(\Omega\setminus\Omega_h)}, & \mbox{ case (b)}.
  \end{aligned}\right.
\end{equation*}
Note that the function $\jmath\sigma_h^\ast$ is continuous and piecewise linear, and thus
the bounded variation norm agrees with the $W^{1,1}(\Omega)$-norm.
Hence we can apply Lemma \ref{lem:femext} to obtain
\begin{equation*}
  \lim\sup_{h\to0} \Psi(\jmath\sigma_h^\ast)-\Psi_h(\sigma_h^\ast) = 0.
\end{equation*}
Hence the sequence $\{\Psi(\jmath\sigma_h^\ast)\}$ is also uniformly bounded, and by Lemma \ref{lem:embed},
there exists a subsequence of $\{\jmath\sigma_h^\ast\}$ such that
$\jmath \sigma_h^\ast \rightarrow \sigma^*$ in $L^1(\Omega)$.
By Lemma \ref{lem:discon}, we have $U_{l,h}(\sigma_h^\ast)\rightarrow U_l(\sigma^\ast)$ as $h \rightarrow 0$.
This together with the weak lower semicontinuity of norms, we deduce that
\begin{equation*}
   \begin{aligned}
    J(\sigma^\ast) &=\tfrac{1}{2}\|U(\sigma^\ast)-U^\delta\|^2 + \alpha\Psi(\sigma^\ast)\\
     &\leq \lim_{h\rightarrow0}\tfrac{1}{2}\|U_h(\sigma_h^\ast)-U^\delta\|^2 + \alpha\liminf_{h\rightarrow0}\Psi(\jmath\sigma_h^\ast)\\
     &\leq\liminf_{h\rightarrow0}\left(\tfrac{1}{2}\|U_h(\sigma_h^\ast)-U^\delta\|^2 + \alpha\Psi_h(\sigma_h^\ast) + \alpha(\Psi(\jmath\sigma_h^\ast)-\Psi_h(\sigma_h^\ast))\right)\\
     & \leq \liminf_{h\rightarrow0}J_h(\sigma_h^\ast) + \limsup_{h\rightarrow 0}\alpha(\Psi(\jmath\sigma_h^\ast)-\Psi_h(\sigma_h^\ast))=  \liminf_{h\rightarrow0}J_h(\sigma_h^\ast),
   \end{aligned}
\end{equation*}

Now we proceed as in the proof of Theorem \ref{thm:polyconv} by considering cases (a) and (b) separately.
For case (a), by the density of $C^\infty(\overline{\Omega})$ in $H^1(\Omega)$, we may assume $\sigma\in C^\infty(\overline{\Omega})\cap \mathcal{A}$,
then $\mathcal{I}_h\sigma \in \mathcal{A}_h$. Further,
\begin{equation*}
   \begin{aligned}
     \|\jmath\mathcal{I}_h\sigma - \sigma\|_{H^1(\Omega)}^2 &= \|\jmath\mathcal{I}_h\sigma - \sigma\|_{H^1(\Omega_h)}^2 + \|\jmath\mathcal{I}_h\sigma - \sigma\|_{H^1(\Omega\backslash\Omega_h)}^2 \\
      &\leq \|\mathcal{I}_h\sigma - \sigma\|_{H^1(\Omega_h)}^2 + 2(\|\jmath\mathcal{I}_h\sigma\|_{H^1(\Omega\backslash\Omega_h)}^2 + \|\sigma\|_{H^1(\Omega\backslash\Omega_h)}^2)\\
      &\leq \|\mathcal{I}_h\sigma - \sigma\|_{H^1(\Omega_h)}^2 + C\epsilon_h \|\mathcal{I}_h\sigma\|_{H^1(\Omega_h)} + 2\|\sigma\|_{H^1(\Omega\backslash\Omega_h)}^2 \rightarrow 0.
\end{aligned}
\end{equation*}
Then by Lemma \ref{lem:discon}, $U_h(\mathcal{I}_h\sigma) \rightarrow U(\sigma)$ and thus
$J(\sigma) = \lim_{h\to 0} J_h(\mathcal{I}_h\sigma) \geq \liminf_{h\to0} J_h(\sigma_h^\ast) = J(\sigma^\ast)$,
i.e., $\sigma^*$ is a minimizer to the continuous functional.

Next consider case (b). For any fixed $\sigma\in\mathcal{A}$, by Lemma \ref{lem:density} and the constructions in the proof of
Theorem \ref{thm:polyconv}, for any $\epsilon>0$, there exists $\sigma^\epsilon\in C(\overline{\Omega})$ such that
\begin{equation*}
  \int_\Omega |\sigma^\epsilon-\sigma|dx <\epsilon\quad\mbox{and}\quad \left|\int_\Omega|\nabla \sigma^\epsilon|-\int_\Omega |D\sigma|\right|<\epsilon.
\end{equation*}
Then $\tilde{\sigma}^\epsilon=P_{[c_0,c_1]}(\sigma^\epsilon)\in W^{1,\infty}(\Omega)\cap\mathcal{A}$. We take $\sigma_h = \mathcal{I}_h\tilde{\sigma}^\epsilon$. Then by Lemma \ref{lem:femext} there holds
\begin{equation*}
   \begin{aligned}
      \|\jmath\mathcal{I}_h\tilde{\sigma}^\epsilon - \tilde{\sigma}^\epsilon\|_{W^{1,1}(\Omega)} &= \|\jmath\mathcal{I}_h\tilde{\sigma}^\epsilon - \tilde{\sigma}\|_{W^{1,1}(\Omega_h)} + \|\jmath\mathcal{I}_h\tilde{\sigma}^\epsilon - \tilde{\sigma}^\epsilon\|_{W^{1,1}(\Omega\backslash\Omega_h)} \\
      &\leq \|\mathcal{I}_h\tilde{\sigma}^\epsilon-\tilde{\sigma}^\epsilon\|_{W^{1,1}(\Omega_h)}+ \|\jmath\mathcal{I}_h\tilde{\sigma}^\epsilon\|_{W^{1,1}(\Omega\backslash\Omega_h)}+ \|\tilde{\sigma}^\epsilon\|_{W^{1,1}(\Omega\backslash\Omega_h)}\\
      &\leq \|\mathcal{I}_h\tilde{\sigma}^\epsilon-\tilde{\sigma}^\epsilon\|_{W^{1,1}(\Omega_h)}+C\epsilon_h \|\mathcal{I}_h\tilde{\sigma}^\epsilon\|_{W^{1,1}(\Omega_h)} + \|\tilde{\sigma}^\epsilon\|_{W^{1,1}(\Omega\backslash\Omega_h)} \rightarrow 0.
\end{aligned}
\end{equation*}
Then by Lemma \ref{lem:discon}, $U_h(\mathcal{I}_h\tilde{\sigma}^\epsilon)\rightarrow U(\tilde{\sigma}^\epsilon)$ and
thus $J(\tilde{\sigma}^\epsilon) = \lim_{h\to0}J_h(\mathcal{I}_h\tilde{\sigma}^\epsilon)
\geq \liminf_{h\to0}J_h(\sigma_h^\ast)=J(\sigma^\ast)$.
The rest is identical with the proof in Theorem \ref{thm:polyconv}. This concludes the proof of the theorem.
\end{proof}

\section{Concluding remarks}\label{sect:concl}

We have provided a convergence analysis of finite element approximations
of the electrical impedance tomography with the popular complete electrode model.
We investigated regularization formulations of Tikhonov type
with either the smoothness or total variation penalty, which represent two most popular imaging
algorithms in practice. The convergence for both polyhedral and convex smooth curved domains has
been established. The latter relies on a careful analysis of the errors incurred by the domain
approximation. This provides partial theoretical justifications of the discretization
strategies. One immediate future problem is the convergence rates analysis, i.e.,
the error between the (discrete) approximation and the true solution in terms
of the noise level and mesh size. This would shed valuable insights into the
practically very important question of designing discretization strategies compatible with
the regularization parameter and noise level so as to effect optimal computational complexity.

\section*{Acknowledgements}
The work was started during a visit of the first author (MG) at Institute of Applied Mathematics
and Computational Science, Texas A\&M University. He would like to thank the institute
for the hospitality. The work of the second author (BJ) was partially
supported by NSF Grant DMS-1319052, and that of the third author (XL) was supported
by National Science Foundation of China No. 11101316 and No. 91230108.
\bibliographystyle{abbrv}
\bibliography{eit}
\end{document}